\documentclass[11pt]{article}
\usepackage{amsthm,amsfonts,amssymb,amsmath,oldgerm}
\numberwithin{equation}{section}
\usepackage{fullpage}
\usepackage{capt-of}
\usepackage{wrapfig}
\usepackage{graphicx}
\usepackage{hyperref}
\usepackage{color}

\usepackage{appendix}
\renewcommand\appendix{\par
\setcounter{section}{0}%
\setcounter{subsection}{0}%
\setcounter{table}{0}
\setcounter{figure}{0}
\setcounter{equation}{0}
\gdef\theequation{\Alph{section}-\arabic{equation}}
\section*{Appendix}
\gdef\thesection{\Alph{section}}
\setcounter{section}{0}} 

\newcommand\br{\begin{remark}}
\newcommand\er{\end{remark}}
\newcommand\bp{\begin{pmatrix}}
\newcommand\ep{\end{pmatrix}}
\newcommand{\be}{\begin{equation}}
\newcommand{\ee}{\end{equation}}
\newcommand\ba{\begin{equation}\begin{aligned}}
\newcommand\ea{\end{aligned}\end{equation}}

\newcommand{\bap}{\begin{app}}
\newcommand{\eap}{\end{app}}
\newcommand{\begs}{\begin{exams}}
\newcommand{\eegs}{\end{exams}}
\newcommand{\beg}{\begin{example}}
\newcommand{\eeg}{\end{exaplem}}
\newcommand{\bpr}{\begin{proposition}}
\newcommand{\epr}{\end{proposition}}
\newcommand{\bt}{\begin{theorem}}
\newcommand{\et}{\end{theorem}}
\newcommand{\bc}{\begin{corollary}}
\newcommand{\ec}{\end{corollary}}
\newcommand{\bl}{\begin{lemma}}

\newcommand{\bd}{\begin{definition}}
\newcommand{\ed}{\end{definition}}
\newcommand{\brs}{\begin{remarks}}
\newcommand{\ers}{\end{remarks}}

\newcommand{\RR}{{\mathbb R}}

\newcommand{\sgn}{\text{\rm sgn}}
\newtheorem{theorem}{Theorem}[section]
\newtheorem{proposition}[theorem]{Proposition}
\newtheorem{corollary}[theorem]{Corollary}
\newtheorem{lemma}[theorem]{Lemma}

\theoremstyle{remark}
\newtheorem{remark}[theorem]{Remark}
\theoremstyle{definition}
\newtheorem{definition}[theorem]{Definition}

\newtheorem{example}[theorem]{Example}

\newcommand{\RM}{\mathbb{R}}

\newcommand{\CM}{\mathbb{C}}

\newcommand{\beq}{\begin{equation}}
\newcommand{\eeq}{\end{equation}}

\newcommand{\spm}{{\ensuremath{{\scriptscriptstyle\pm}}}}
\renewcommand{\sp}{{\ensuremath{{\scriptscriptstyle +}}}}
\newcommand{\sm}{{\ensuremath{{\scriptscriptstyle -}}}}

\global\long\def\norm#1{\left\Vert #1\right\Vert }

\title{
On the Dynamics of Traveling Fronts Arising in Nanoscale Pattern Formation
}

\author{
Mathew A. Johnson\thanks{Department of Mathematics, University of Kansas, 1460 Jayhawk Boulevard, 
Lawrence, KS 66045; matjohn@ku.edu}\and Gregory D. Lyng\thanks{UHG R\&D, UnitedHealth Group, 5995 Opus Parkway, Minnetonka, MN 55343; glyng@savvysherpa.com.}\and Connor Smith \thanks{Department of Mathematics, University of Kansas, 1460 Jayhawk Boulevard, Lawrence, KS 66045; c406s460@ku.edu.}
}

\begin{document}
\maketitle

%
%
%

\vspace{-1em}
\begin{abstract}
We study the stability and dynamics of traveling-front solutions of a modified Kuramoto--Sivashinsky equation arising in the modeling of nanoscale ripple patterns that form when a nominally flat solid surface is bombarded with a broad ion beam at an oblique angle of incidence.
Structurally, the linearized operators associated with these fronts have unstable essential spectrum---corresponding to instability of the spatially asymptotic states---and stable point spectrum---corresponding to stability of the transition profile of the front. We show that these waves are linearly orbitally asymptotically stable in appropriate exponentially weighted spaces.  
While the technical device of exponential weights allows us to accommodate the unstable essential spectrum of individual waves in our linear analysis, it does not shed light on the long-time pattern formation that is observed experimentally and in numerical simulations. To begin to address this issue, we consider a periodic array of unstable front and back solutions. While not an exact solution of the governing equation, this periodic pattern mimics experimentally observed phenomena. Our numerical experiments suggest that the convecting instabilities associated with each individual wave are damped as they pass through transition layers and that this stabilization mechanism underlies the pattern formation seen in experiments.

\end{abstract}

\tableofcontents

\section{Introduction}

\subsection{Terraces in Nanoscale Pattern Formation}
The impact of an energetic ion on a solid flat surface may dislodge and eject one or more atoms from the surface. Thus, if a broad ion beam is continuously used to pepper the solid with energetic ions, the surface will erode. Depending on a variety of factors (e.g., incidence angle, ion energy, surface composition), a surprising array of patterns can self assemble on the surface due to this process. For example, periodic height modulations (\emph{ripples}), mounds arranged in regular hexagonal arrays, and regular configurations of pointed conical protrusions have all been experimentally observed. 
The patterns can have length scales as small as $10\;\mathrm{nm}$ \cite{MSB_JPD12}.

Here, we focus on ripple formation, and, in particular, we study a model introduced by Pearson \& Bradley \cite{pearson_theory_2014} to capture late-stage terrace formation. That is, experiments often show that the late-time profile of the surface alternates in a sawtooth fashion between two nearly constant-slope segments with slopes $m_\sp>0$ and $m_\sm<0$. Since, typically, $|m_\sp|\neq |m_\sm|$, the sawtooth pattern is asymmetric; see Figure \ref{f:topographyschematic}. Pearson \& Bradley were particularly interested in (i) modeling terrace formation and (ii) the selection process for the slopes $m_\spm$. Standard models for the unstable surface take the form of an anisotropic Kuramoto--Sivashinsky equation with a quadratic nonlinearity in two space dimensions. 
However, these models do not support terrace formation, so Pearson \& Bradley introduced a model that takes account of higher-order nonlinear processes, and they showed that this model, a Kuramoto--Sivashinsky equation with a cubic nonlinearity (see \eqref{height_eqn}), does indeed support terrace formation. Moreover, they showed that their model also captures slope selection (i.e., the choices of $m_\spm$); more precisely, they showed that these values are the asymptotic end states of undercompressive shock profiles which arise as solutions of a related differential equation. The undercompressive nature of these profiles manifests itself in the traveling-wave equation. In this case, the desired heteroclinic connection between equilibria arises as the structurally unstable intersection of a one-dimensional stable manifold and a two-dimensional unstable manifold in $\RR^3$. Thus, we expect that such solutions will only exist for distinguished values of the parameters, and this imposes a kind of selection on the existence problem. See Section \ref{s:front} for further discussion on this point.

\subsection{Dynamical Behavior and Stability: Overview of Results}
Here, we study the stability and dynamical behavior of the undercompressive shock profiles that describe the slope transitions in terrace formation. After reviewing the Pearson--Bradley model, we begin in 
Section \ref{s:front} by revisiting the existence problem. As noted above, this is a connecting orbit problem for a nonlinear dynamical system in a three-dimensional phase space. We take this problem, originally posed as a two-point boundary-value problem on the line, and we approximate it, by means of appropriate projective boundary conditions, as a two-point boundary-value problem on a finite interval. This latter problem is amenable to solution via MATLAB's built-in BVP solver. 

We next numerically investigate the local dynamical behavior of these profiles in Section \ref{s:num_dynamics_study}. Generally speaking, these experiments suggest the following response  of the system to a small perturbation. First, up to a phase shift, perturbations pass through and decay in the transition layer connecting $m_\sp$ and $m_\sm$. However, any disturbances which reach the asymptotically constant portion of the traveling wave clearly trigger an instability that saturates at some $O(1)$ magnitude but whose spatial extent grows in time. This precludes any kind of stability in translationally invariant Lebesgue or Sobolev norms.   

However, we are able to capture, at the linear level, the dynamics 
exhibited by the transition layers. In particular, we introduce \emph{weighted norms} and we are able to show linear orbital asymptotic stability in appropriate exponentially weighted function spaces.  In particular, we see that, at least at the linear level, these front solutions are convectively unstable \cite{sandstede_absolute_2000}.   This, including the associated spectral analysis, is one of the principal contributions of this paper, and this work is reported in Section \ref{s:Essential-Instability} and Section \ref{sec:linear}. At the spectral level the dichotomy between the stable transition layer and unstable end states is clearly reflected in unstable essential spectrum and stable point spectrum of the associated linearized operator.

Finally, the fact that stable terrace patterns form in the long-time dynamics suggests that these patterns may be attractors for the global dynamics of any mathematical model which reasonably incorporates the underlying physics of the ion bombardment process. Alternatively, these patterns might be \emph{metastable} structures since over very long time scales there is often an observed coarsening of the patterns in experiment. Regardless, the experimental picture needs to be reconciled with the instability of the individual fronts described here, and we begin to do so in Section \ref{sec:stabilization}. In particular, we construct ad hoc periodic patterns of front and back solutions, and we study the $L^2(\RR)$ spectrum of the periodic-coefficient linear operators which result from substituting the $x$-dependent pattern into the form of the operator linearized about one of the profiles.  Note while the construction of these ad-hoc periodic waves is in some way similar to that considered in \cite{SS2000},
we emphasize that here our heteroclinic orbits connect two unstable end states, while in \cite{SS2000} the authors considered connections between one stable and one unstable end 
state\footnote{In this case, it was shown that pulses constructed by concatenating such (dynamically) 
unstable heteroclinic orbits -- while controlling how much time is spent near the unstable end state -- could result in the formation
of stable patterns.  In the present case, since \emph{both} end states are unstable, it is rather a \emph{periodic array} of pulses that result in a stable pattern.}.
Our resulting spectral and time-evolutionary numerical calculations support spectral and nonlinear stability in \emph{unweighted} spaces, and this finding aligns with the experimental results. 
Specifically, we see that a such wave trains built by unstable profiles can stabilize each other by de-amplification of convective instabilities as they pass through successive
transition layers.  This suggests that there should exist nearby periodic wave trains that are stable, despite the instability of the individual front or back profiles.  We regard proofs of  the existence of spectrally and nonlinearly stable periodic solutions as fundamental open problems in the mathematical theory associated with the model \eqref{height_eqn}.  Similar observations have been made in the context of inclined thin-film flow: see \cite{PSU,BJRZ11,barker_nonlinear_2013}, for example.

%


%
%

\begin{figure}[t]
\begin{center}
\includegraphics[scale=0.7]{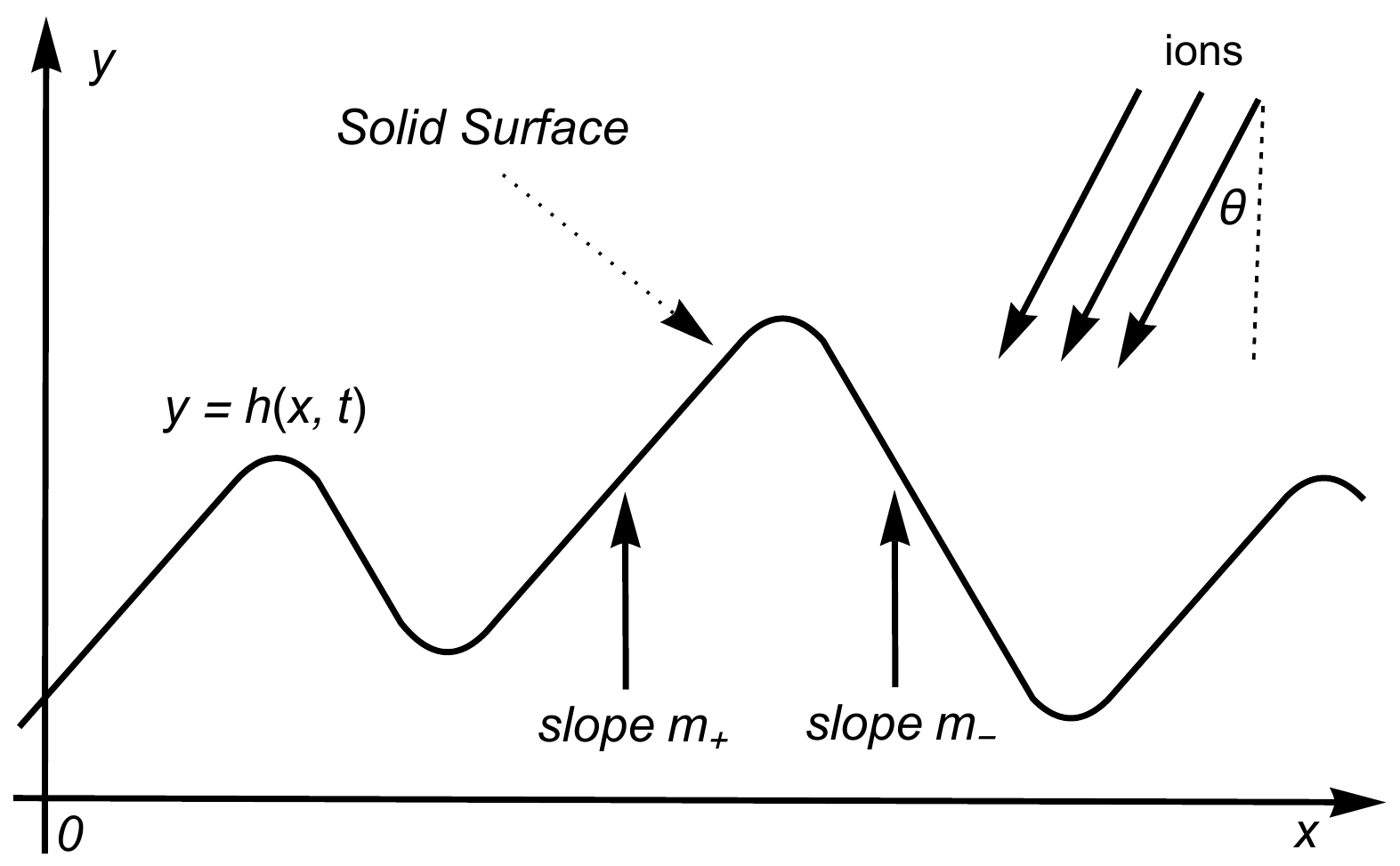}
\end{center}
\caption{Schematic picture of the terraced topography arrising from an ion incidence angle of $\theta$.}
\label{f:topographyschematic}
\end{figure}

\subsection{Equation of Motion}\label{s:model}

We begin our investigation by recalling the equations of motion introduced by Pearson \& Bradley \cite{pearson_theory_2014}. In non-dimensional form, the model for the surface height $h(x,t)$ of an ion-bombarded elemental material takes the form\footnote{In the remainder of this paper, we use subscripts to denote partial derivatives.}
\begin{equation}\label{height_eqn}
\frac{\partial h}{\partial t}=-v_0-\frac{\partial^2h}{\partial x^2}-\frac{\partial^4h}{\partial x^4}+\frac{1}{2}\left(\frac{\partial h}{\partial x}\right)^2+\frac{\gamma}{6}\left(\frac{\partial h}{\partial x}\right)^3\,,
\end{equation}
where here $\gamma>0$ is a dimensionless parameter that depends on the angle of incidence $\theta$ and the material parameters, and $v_0>0$ is a downward ``drift" speed accounting 
for erosion of the background elemental surface on which the terraced patterns are expected to form.  Figure \ref{f:topographyschematic} contains a schematic diagram of the configuration.
Of course, this model assumes the surface height is independent of the transverse spatial
variable of the surface.  From this model for the height $h(x,t)$, we briefly recall the derivation of Pearson \& Bradley to obtain a model for the slope of the surface. 

To describe patterns forming on the eroding  surface, we first subtract off the background erosion by decomposing the solution as
$h(x,t)=-v_0t+u(x,t)$ and note that the relative height $u$ satisfies 
\[
u_t=-u_{xx}-u_{xxxx}-\frac{1}{2}\left(u_x\right)^2+\frac{\gamma}{6}\left(u_x\right)^3\,.
\]
The quadratic nonlinearity above can be removed by making the transformation
\[
P(x,t):=\sqrt{\frac{\gamma}{6}}\left[u(x,t)-\frac{x}{\gamma}+\frac{t}{3\gamma^2}\right]\,.
\]
Indeed, an elementary calculation shows that the transformed height $P$ satisfies
\[
P_t=-P_{xx}-P_{xxxx}+\left(P_x\right)^3-\frac{1}{2\gamma}P_x\,.
\]
Changing to the moving coordinate frame $\tilde{x}=x-\frac{1}{2\gamma}t$ we find, after dropping tildes for notational convenience,
\[
P_t=-P_{xx}-P_{xxxx}+\left(P_x\right)^3.
\]
Setting $p(x,t)=\partial_xP(x,t)$, we find the transformed slope of the eroding surface satisfies the conservative partial differential equation
\begin{equation}\label{mks}
p_t=-p_{xx}-p_{xxxx}+(p^3)_x,
\end{equation}
which can be recognized as a Kuramoto-Sivashinsky-type model with a cubic\footnote{We will refer to \eqref{mks} as the modified Kuramoto-Sivashinsky equation.} nonlinearity.

Our initial goal is to study the existence and stability of traveling front solutions in the transformed slope model \eqref{mks}.  Using well-conditioned numerical methods, 
we will construct a one-parameter family of front solutions to \eqref{mks}.  Note that original slope $h_x$ and transformed slope $p$ of the surface are related via the relation
\[
p(x,t)=\sqrt{\frac{\gamma}{6}}\left(h_x(x,t)-\frac{1}{\gamma}\right)
\]
so that, in particular, an asymptotically constant solution $p$ of \eqref{mks} with symmetric end states, which will be the primary (but not only) object of our study, 
corresponds to an asymmetric slope pattern on the original bombarded surface.  As mentioned in the introduction, this is consistent with expectations from experiment \cite{pearson_theory_2014}.

\

\subsection{Outline, Main Results \& Open Problems}

An outline of the paper is as follows.
We begin in Section \ref{s:num_exist} with a numerical study of the existence traveling front solutions of modified Kuramoto-Sivashinsky equation \eqref{mks}.
We will then consider the dynamics of such traveling fronts when subject to small localized (i.e. integrable) perturbations: see Section \ref{s:num_dynamics_study}.  In
Section \ref{s:Essential-Instability}, we aim to mathematically capture and describe the numerical observations regarding the local dynamics near traveling
front solutions.  Our main result in this direction is Theorem \ref{T:main_lin}, which establishes the linear orbital asymptotic stability of such solutions in appropriately weighted spaces: see
also Corollary \ref{C:main_lin}, which states the main results in terms of the original unweighted solution.
An argument upgrading these results to a nonlinear level has so far evaded us, and in Section \ref{s:nonlinear} we detail the challenges faced and its relation to other works. 
We consider the resolution of this nonlinear analysis to be the fundamental open problem posed in this paper.

Finally, in Section \ref{sec:stabilization} we numerically construct ad hoc periodic traveling wave solutions of \eqref{mks} and demonstrate that such patterns seem
to be spectrally and  nonlinearly stable in unweighted spaces.   In particular, we make the following interesting observation: there exists a critical period $X_{\rm crit}$ such that
the periodization of a single ``pulse", formed by concatenating a front with a back solution, is spectrally stable for periods $X<X_{\rm crit}$ while they are unstable for $X>X_{\rm crit}$.
Moreover, if one periodizes $n$-``pulses", formed by concatenating $n$-fronts with $n$-backs, then the resulting oscillatory waveform becomes unstable if the period
of the \emph{entire pattern} is greater than $nX_{\rm crit}$, while it is stable otherwise.  This implies instabilities that are so large that they would overtake a single periodized pulse
can be stabilized by placing a sufficient number of front-back ``pulses" in its wake periodically.  We consider the rigorous establishment of this nonlinear phenomena regarding periodic
waves as an important and interesting open problem

\section{Numerical Existence and Dynamics of Traveling Fronts}\label{s:front}

In an effort to understand the terraced slope selection, Pearson \& Bradley \cite{pearson_theory_2014} sought traveling front solutions of the transformed slope equation \eqref{mks}.
In this section, we numerically investigate the existence of such fronts as well as their local dynamics under perturbation.  Through our time evolution
studies, a clear picture seems to form concerning the behavior of such fronts when subject to small localized (i.e., integrable on $\RM)$ perturbations.
We will then work in forthcoming sections to mathematically justify the observed local dynamics about such fronts.

\subsection{Numerical Existence Study}\label{s:num_exist}

We begin by seeking traveling wave solutions of the transformed slope equation \eqref{mks}, i.e. solutions of the form $p(x,t)=\phi(x-st)$ for some
profile $\phi$ and wave speed $s\in\RM$.
The profile $\phi(z)$ is readily seen to be a stationary solution of the evolutionary equation
\begin{equation}\label{travel_pde}
p_t=sp_z-p_{zz}-p_{zzzz}+(p^3)_z.
\end{equation}
Stationary solutions $\phi$ of \eqref{travel_pde}, after performing a single integration, necessarily satisfy the third-order ODE 
\begin{equation}\label{profile}
\phi'''=s\phi-\phi'+\phi^3+\mu,
\end{equation}
where here $\mu$ is a constant of integration, and primes denote differentiation with respect to the traveling $z=x-st$ coordinate.
Traveling front solutions of \eqref{mks} now correspond to heteroclinic orbits of the ODE \eqref{profile} in its natural three-dimensional phase space, 
i.e. solutions $\phi(z)$ that are asymptotically constant
to \emph{different} asymptotic end states at the spatial infinities.  From \eqref{profile}, it is clear that the values of the aysmptotic end states
correspond to the roots of the polynomial equation
\begin{equation}\label{eq}
\phi^3+s\phi+\mu=0.
\end{equation}
Clearly, this cubic polynomial has only one real root when $s\geq 0$, while when $s<0$ there exists three real roots $\phi_-<\phi_0<\phi_+$, depending on $s$ and $\mu$, with
$\sgn(\phi_0)=\sgn(\mu)$.  Furthermore, linearizing \eqref{profile} about these equilibria show that their spectral stability is determined according
to the roots of the polynomial equation
\[
\lambda^3+\lambda-\left(s+3\phi^2\right)=0,
\] 
evaluated at the respective root $\phi_{\pm}$ or $\phi_0$.
Since the above equation is strictly increasing as a function of $\lambda$, it is clear that there can be only one real root $\lambda_1$ with 
$\sgn(\lambda_1)=\sgn(s+3\phi^2)$, while the other two roots appear as a complex conjugate pair $\lambda_{\pm}$ with 
$\Re(\lambda_{\pm})=-\frac{1}{2}\lambda_1$.  It follows that so long as
\begin{equation}\label{exist_ineq}
s+3\phi_{\pm}^2>0,
\end{equation}
$\phi_{\pm}$ are hyperbolic equilibria of the profile ODE \eqref{profile}, each having a one-dimensional
unstable manifold and a two-dimensional stable manifold corresponding to a stable spiral.  From Figure \ref{fig:mu_exist}, it appears there exists a critical $\mu_*\approx 0.24679$ such
that \eqref{exist_ineq} holds for all $\mu<\mu_*$.  In this parameter range, we now numerically search for heteroclinic orbits 
connecting $\phi_-$ to $\phi_+$.

\begin{figure}[t]
\begin{center}
\includegraphics[scale=0.5]{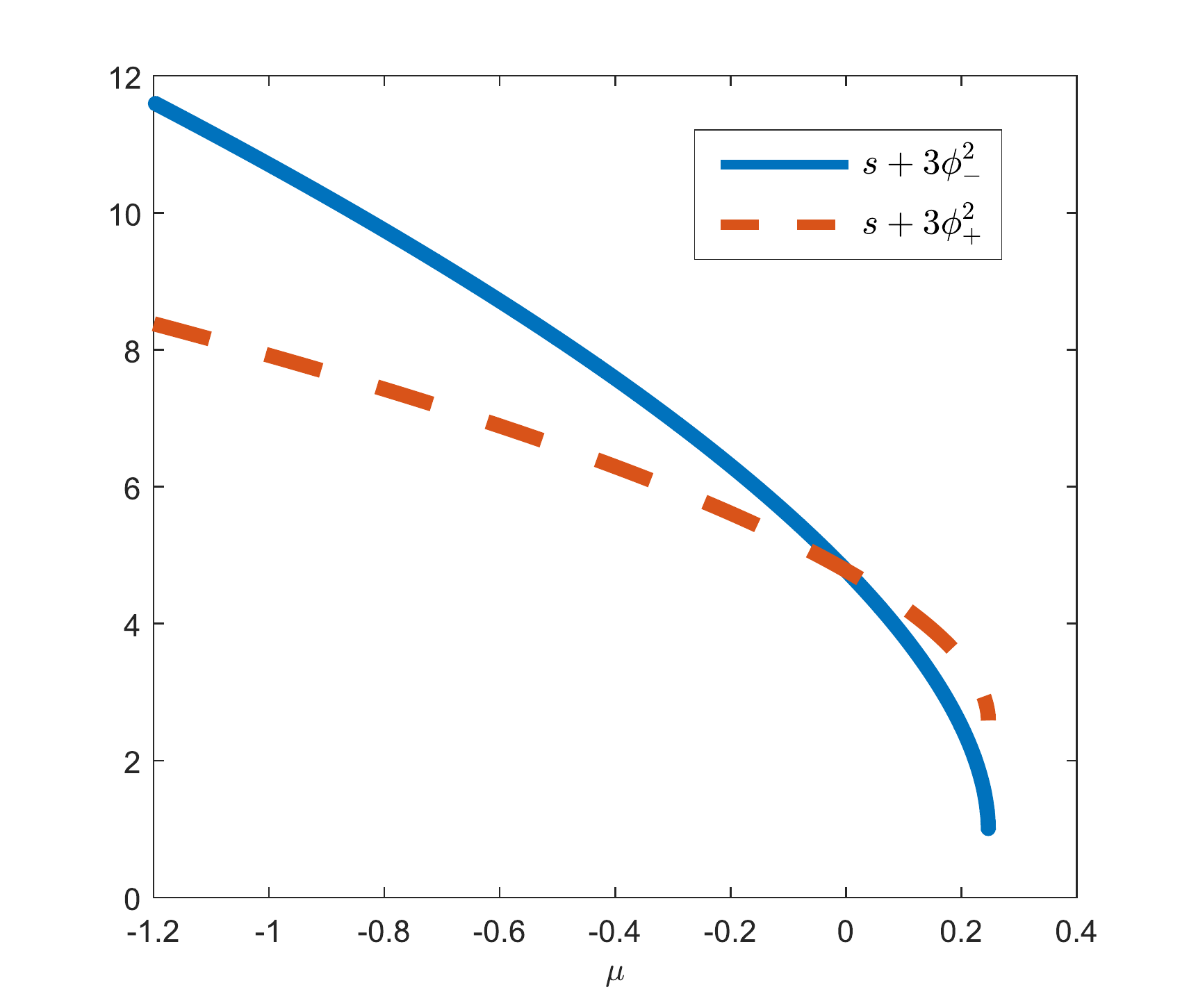}
\caption{A numerical plot of the functions $\mu\mapsto s+3\phi_{\pm}^2$.  It appears there exists a critical $\mu_*\approx 0.24679$ such that both of these maps are strictly
positive for all $\mu<\mu_*$.  Numerically, we find our continuation method of constructing fronts for $\mu\neq 0$ breaks down as $\mu$ approaches $\mu_*$.}\label{fig:mu_exist}
\end{center}
\end{figure}

Ultimately we will phrase this search as solving a boundary value
problem and use a collocation method to solve it. However, that collocation
method requires an initial guess which we generate by using the shooting
method. To explain how the shooting method is applicable here, we can first
idealize a heteroclinic orbit connecting $\phi_{-}$ to $\phi_{+}$ as an
initial value problem with initial condition somewhere in the unstable
subspace of $\phi_{-}$, $E_{\phi_{-}}^{u}$, that at some time enters the
stable subspace of $\phi_{+}$, $E_{\phi_{+}}^{s}$. Note that $E_{\phi_{-}}^{u}$
is one-dimensional so it can be described as the span of a vector
$\vec{e}_{\phi_{-}}^{u}$. Then for the shooting method we introduce
a parameter $\epsilon$ and attempt to optimize its value so that
the initial condition $\phi_{-}+\epsilon~\vec{e}_{\phi_{-}}^{u}$ results
in a solution that comes as close to $\phi_{+}$ and $E_{\phi_{+}}^{s}$as
possible.

When implementing this numerically we take $\mu=0$, $s=-2.388$,
the value for $s$ as recommended in \cite{pearson_theory_2014},
and used MATLAB's built-in \texttt{ode45} to solve the initial value
problem over the spatial range $\left[-7,10\right]$. We used MATLAB's
built-in optimizer \texttt{fminsearch} to find the $\epsilon$ which
minimizes the difference between the value of the solution at $x=10$
and $\phi_{+}$. The optimizer settled on the value $\epsilon=1.2020\times10^{-8}$
which is shown in Figure \ref{fig:Shooting-Method-Front}(a). 
While the result seems a reasonable approximation for a front solution of \eqref{profile}, we find it useful in our forthcoming time evolution studies
to further refine this approximation as follows.


\begin{figure}[t]
\begin{centering}
(a)\includegraphics[scale=0.4]{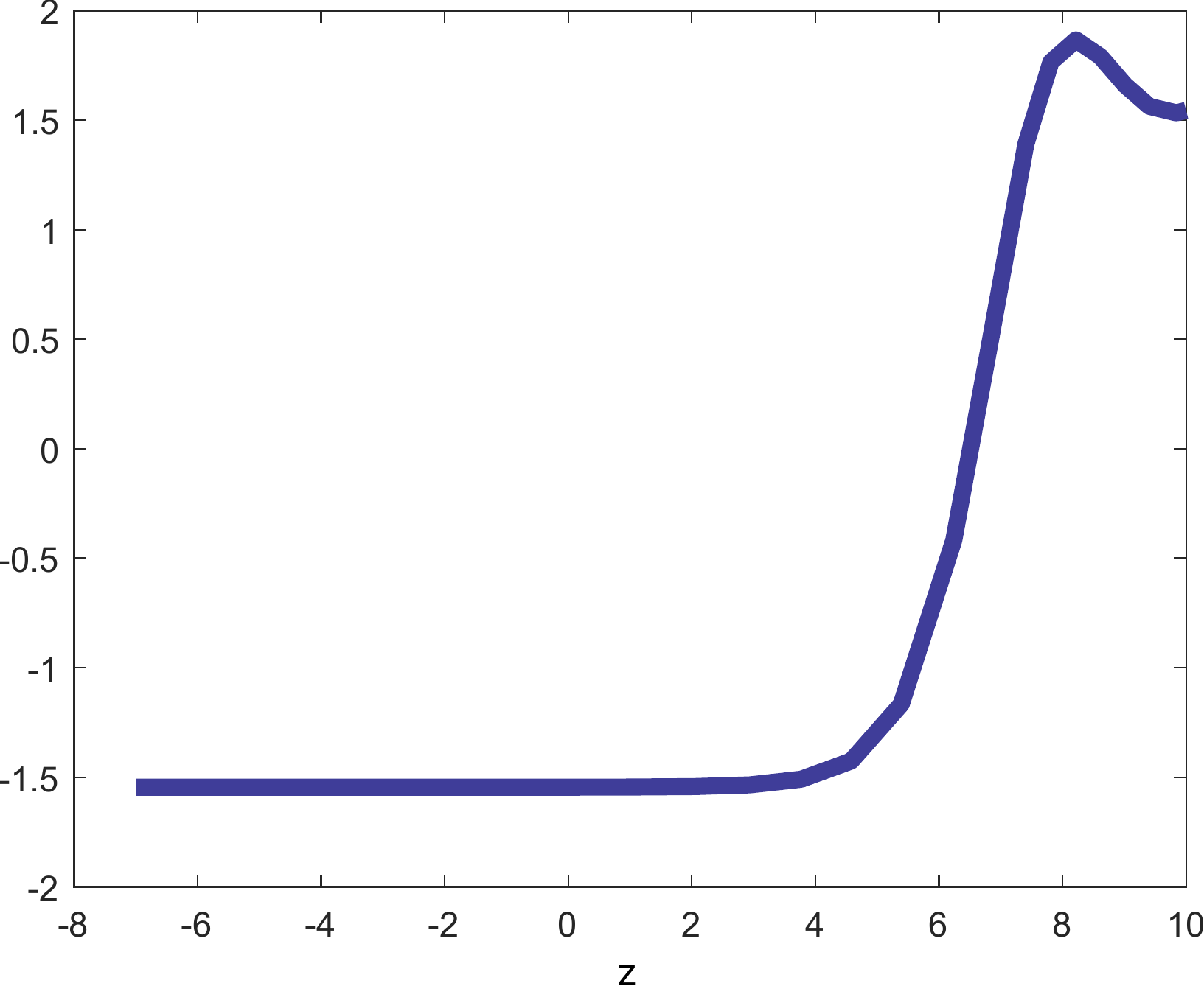}$\qquad$(b)\includegraphics[scale=0.4]{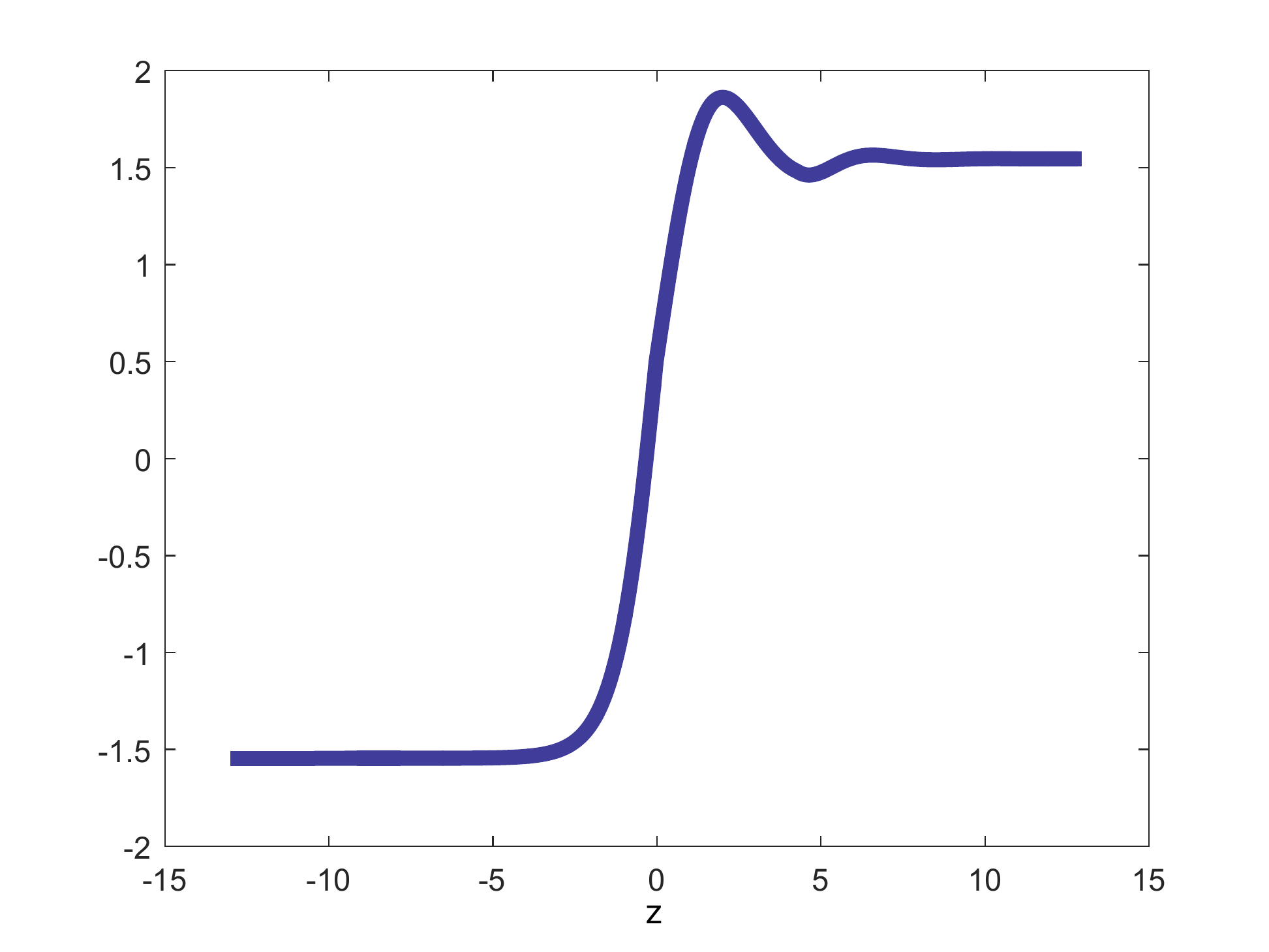}
\par\end{centering}
\caption{(a) The result from using the shooting method to solve \eqref{profile} with $\mu=0$.  Here, as in \cite{pearson_theory_2014}, we find a heteroclinic
orbit for $s\approx-2.388$. (b) Refined numerical approximation of the front in (a) resulting from the collocation method.}\label{fig:Shooting-Method-Front}
\end{figure}

To refine our front approximation we utilize a collocation method, in which we idealize finding
a heteroclinic orbit connecting $\phi_{-}$ and $\phi_{+}$ as a boundary
value problem\footnote{Before using this front as an
initial guess for the collocation method, we extend the front to the
right by padding with its rightmost value in order to ``center''
the front.}. A naive guess would be to use the actual values of
$\phi_{-}$ and $\phi_{+}$ as boundary values: in practice this tends not
to work well. Instead we use projective boundary conditions, where
at the $\phi_{-}$ boundary we require the solution to be in $E_{\phi_{-}}^{u}$
and at the $\phi_{+}$ boundary we require the solution to be in $E_{\phi_{+}}^{s}$.
We implement the collocation method using MATLAB's built-in \texttt{bvp4c}
which treats boundary conditions as quantities to be minimized. We
can numerically calculate the complementary spaces $E_{\phi_{-}}^{s}$
and $E_{\phi_{+}}^{u}$by linearizing \eqref{profile} about $\phi_{-}$ and
$\phi_{+}$, writing each as a system, then finding the stable and unstable
eigenvectors respectively. Then after using Gram-Schmidt to ensure
each vector is orthogonal, we have $E_{\phi_{-}}^{s}=\text{span}\left(\vec{e}_{\phi_{-}}^{~s,1},\vec{e}_{\phi_{-}}^{~s,2}\right)$
and $E_{\phi_{+}}^{u}=\text{span}\left(\vec{e}_{\phi_{+}}^{~u}\right)$.
At the $\phi_{-}$ boundary we would minimize the inner product of the
solution with $\vec{e}_{\phi_{-}}^{~s,1}$ and $\vec{e}_{\phi_{-}}^{~s,2}$,
and at the $\phi_{+}$ boundary we would minimize the inner product of
the solution with $\vec{e}_{\phi_{+}}^{~u}$.

However, we have the lingering issue of translation invariance: that
the equation \eqref{profile} is translation invariant and that the
above projective boundary conditions do not fix a particular translate.
To resolve this, we use a trick from \cite{ghazaryan_spectral_2013}
that allows us to specify a particular function value in the ``middle''
of the region in exchange for doubling the dimension of the problem.
Specifically, that if \eqref{profile} when written as a system is $\partial_{z}y=\mathcal{F}\left(y\right)$,
then we introduce the variable $\tilde{y}$ with $\partial_{z}\tilde{y}=-\mathcal{F}\left(\tilde{y}\right)$
so rather than working on the region $\left[-L,L\right]$, $y$ is
the solution on $\left[0,L\right]$ and $\tilde{y}$ is the solution
running backwards on $\left[-L,0\right]$. This allows us to have
the additional three-dimensional matching condition that $y\left(0\right)=\tilde{y}\left(0\right)$
within which we can specify that the front takes on the (arbitrarily
chosen) value $\frac{1}{2}$ at $z=0$.

With that the boundary conditions have been established. For an initial
guess we take the output from the shooting method. We have the fixed
parameter $\mu$ as an input. Regarding $s$, the \texttt{bvp4c} implementation
in MATLAB allows for $s$ to be treated as a parameter to be solved
for. This even accounts for the fact that changing $s$ changes the
boundary conditions. Running this with $s=-2.388$ and $\mu=0$
produces the a refinement of the front produced from our shooting method: see Figure \ref{fig:Shooting-Method-Front}(b).

The choice of $\mu=0$ in the above may seem arbitrary, but choice
of $\mu$ is intertwined with the issue of creating ``back'' solutions.
That is, so far we have been numerically approximating heteroclinic
orbits from $\phi_{-}$ to $\phi_{+}$ and calling them front solutions.
However, through an analogous procedure we could have constructed ``back" solutions,
which connect the equilibria $\phi_+$ at $-\infty$ to $\phi_-$ at $+\infty$: see Figure \ref{fig:Choosing-mu}(b). 
This construction will be crucial later in Section \ref{sec:stabilization} where
we investigate the stability and dynamics of periodic structures formed by concatenating a front solution
and a back solution.  Naturally, to ensure this pattern is a stationary solution of \eqref{travel_pde}, we must
ensure the  corresponding front and back solutions being concatenated have the same speed.
In Figure \ref{fig:Choosing-mu}(a) we show the dependence of the wave speed $s$, determined numerically through continuation,
on the parameter $\mu$.  
From these numerics, it seems clear that\footnote{Alternatively, observe the profile equation \eqref{profile} is invariant under the transformation $(\phi,s,\mu)\mapsto(-\phi,s,-\mu)$.}
$s_{\rm front}(\mu)=s_{\rm back}(-\mu)$ and hence the only way to ensure the front and back solutions have the same speed
is to have them correspond to $\mu$ values with opposite signs.
However, from \eqref{eq} we know that changing $\mu$ also changes the asymptotic endstates of
profile $\phi$: see, for example,  one particular example demonstrated in Figure \ref{fig:Choosing-mu}(b). 
In particular, we find the only way to construct front and back solutions with same speed and the same asymptotic endstates is to require $\mu=0$.  See additional discussion
in Section \ref{sec:stabilization}.

\begin{figure}[t]
\begin{centering}
(a)\includegraphics[scale=0.4]{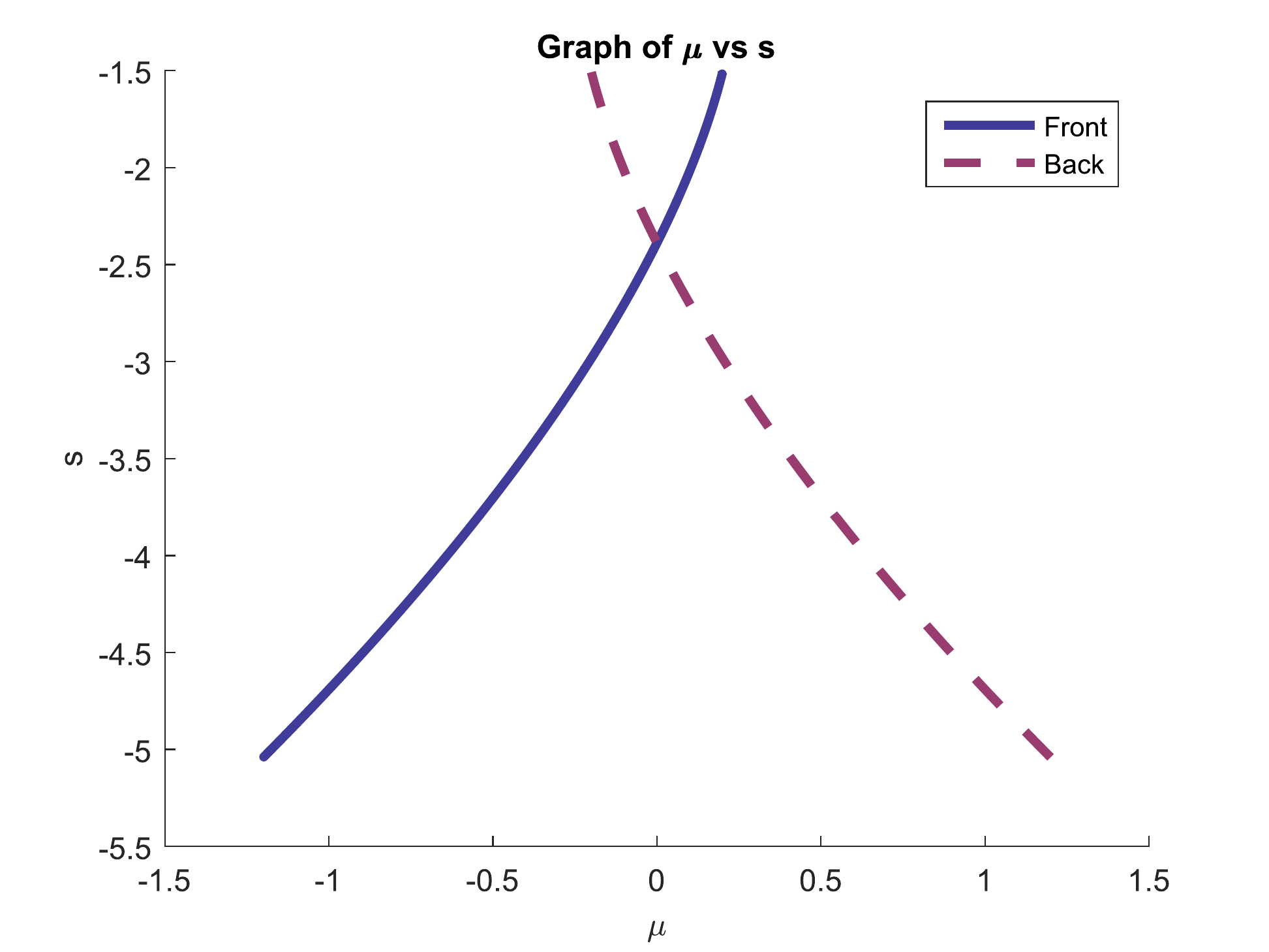}$\qquad$(b)
\includegraphics[scale=0.43]{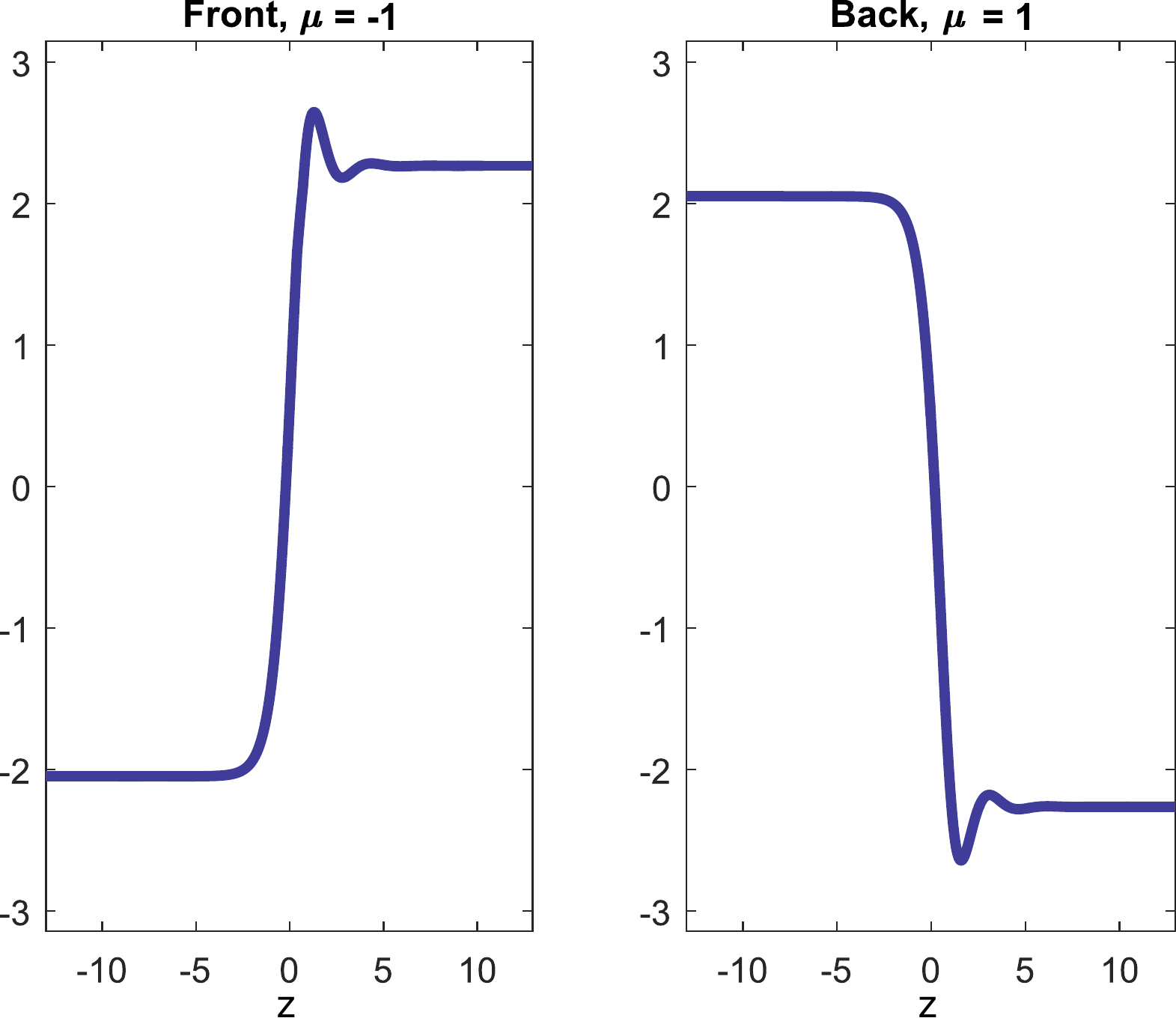}
\caption{ (a) A numerically generated graph of the dependence of the wave speed $s$, as determined by the collocation method, for both front and back solutions
on the parameter $\mu$ in \eqref{profile}.  Recall such fronts will only exist for $\mu<\mu_*\approx 0.24679$, while backs will exist only for $\mu>-\mu_*$: see Figure \ref{fig:mu_exist}.  
(b) Numerically generated front solutions with $(\mu,s)=(-1,s(-1))$,  together with a back solution corresponding to $(\mu,s)=(1,s(1))$, where here $s(1)=s(-1)\approx -4.69$.
Observe that while these profiles have the same wavespeed, they have different asymptotic end states.  
}\label{fig:Choosing-mu}
\end{centering}
\end{figure}

\subsection{Numerical Local Dynamics Study}\label{s:num_dynamics_study}

Given a traveling front solution $\phi(z)$ with wave speed $s<0$, we are now interested in the behavior of solutions
of the initial value problem
\begin{equation}\label{IVP_pert}
\left\{\begin{aligned}
&p_t=sp_z-p_{zz}-p_{zzzz}+\left(p^3\right)_z\\
&p(0,z)=\phi(z)+v(z),
\end{aligned}\right.
\end{equation}
where here $v$ is a sufficiently small function in an appropriately smooth subspace of $L^2(\RM)$.  The stationary solution $\phi$ is said to be 
stable if solutions that start ``initially close" to $\phi$ stay ``close" to $\phi$ for all time.  In other words, if all solutions of the above
initial value problem stay ``close" to $\phi$, so long as $v$ is sufficiently ``small" in some sense, then $\phi$ is said to be stable.  Otherwise, the wave is said to be unstable.
Naturally, the notion of stability is intimately related to the notion of ``closeness" used.  In order to determine if the traveling waves
constructed above may be stable in some sense, we now study the above initial value problem numerically.

Once the pair $(\phi,s)$ and the initial perturbation $v$ are chosen, we numerically
solve \eqref{IVP_pert} by following the method used in \cite{barker_nonlinear_2013}.  Specifically, we use
a Crank-Nicolson discretization in time together with Newton's method to
solve the nonlinear system of equations resulting from discretization.  After running this computational scheme
for various waves $\phi$, hence different $\mu$, as well as different initial perturbations $v$, a common dynamical
picture formed.  Specifically, the effect of the initial perturbation always seems to move to the left (in the moving coordinate frame).  As the perturbation
evolves through the ``transition" part of the solution (the sort of regularized shock), the location of this transition shifts leftward by a small amount while the size of the perturbation
seems to decrease drastically while passing through this transition. 
Furthermore, after the perturbation has passed through the transition, we are left with what looks like a very small
spatial translation of the initial transition part of the wave.  In this sense, it seems that at least the transition portion of the wave exhibits a sort of dynamical stability; see
Figure \ref{evolve_front_pert}(a)-(b) for the case $\mu=0$.

\begin{figure}[t]\begin{centering}
(a)\hspace{-.3em}\includegraphics[scale=0.29]{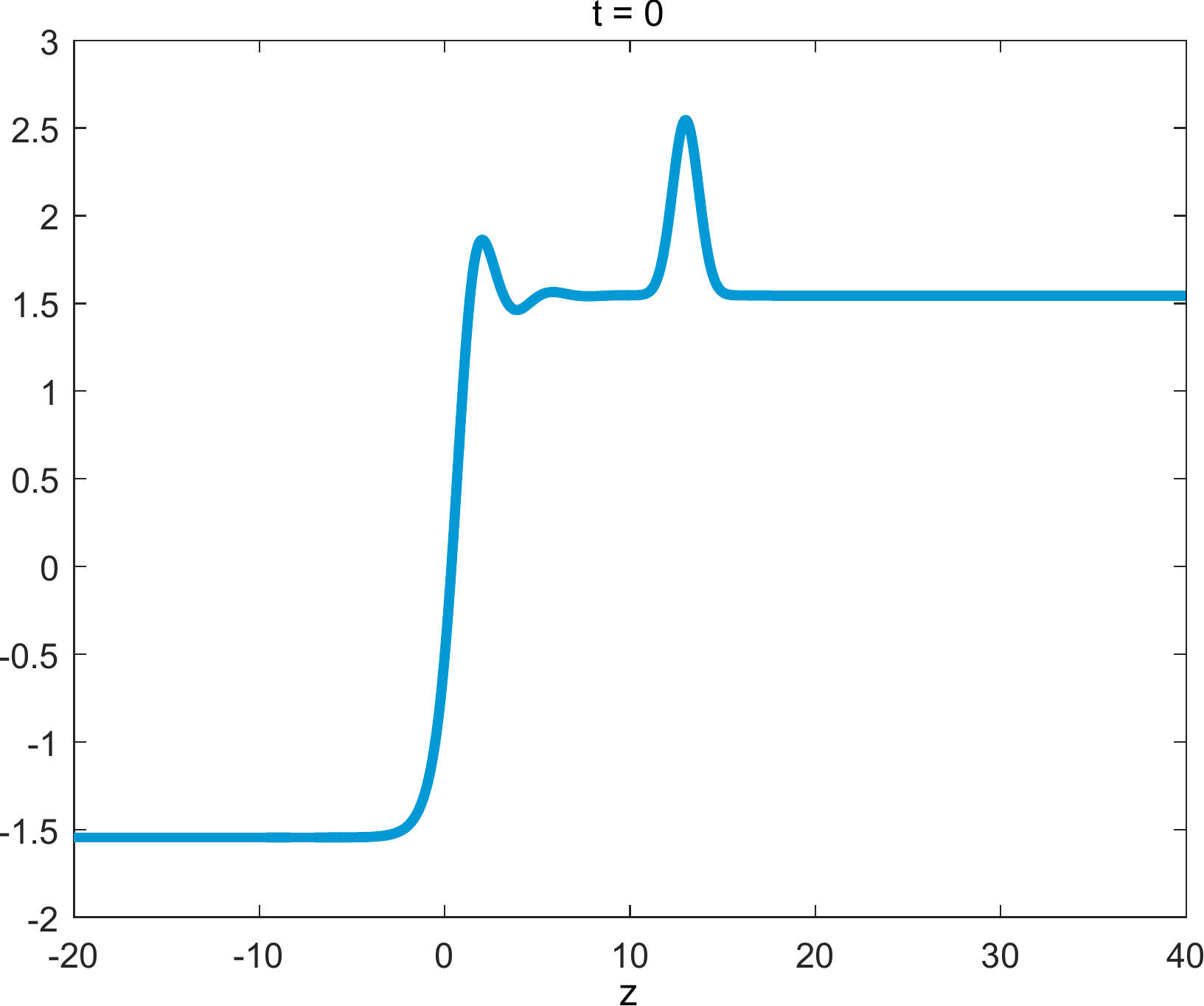}
~~(b)\hspace{-.3em}\includegraphics[scale=0.29]{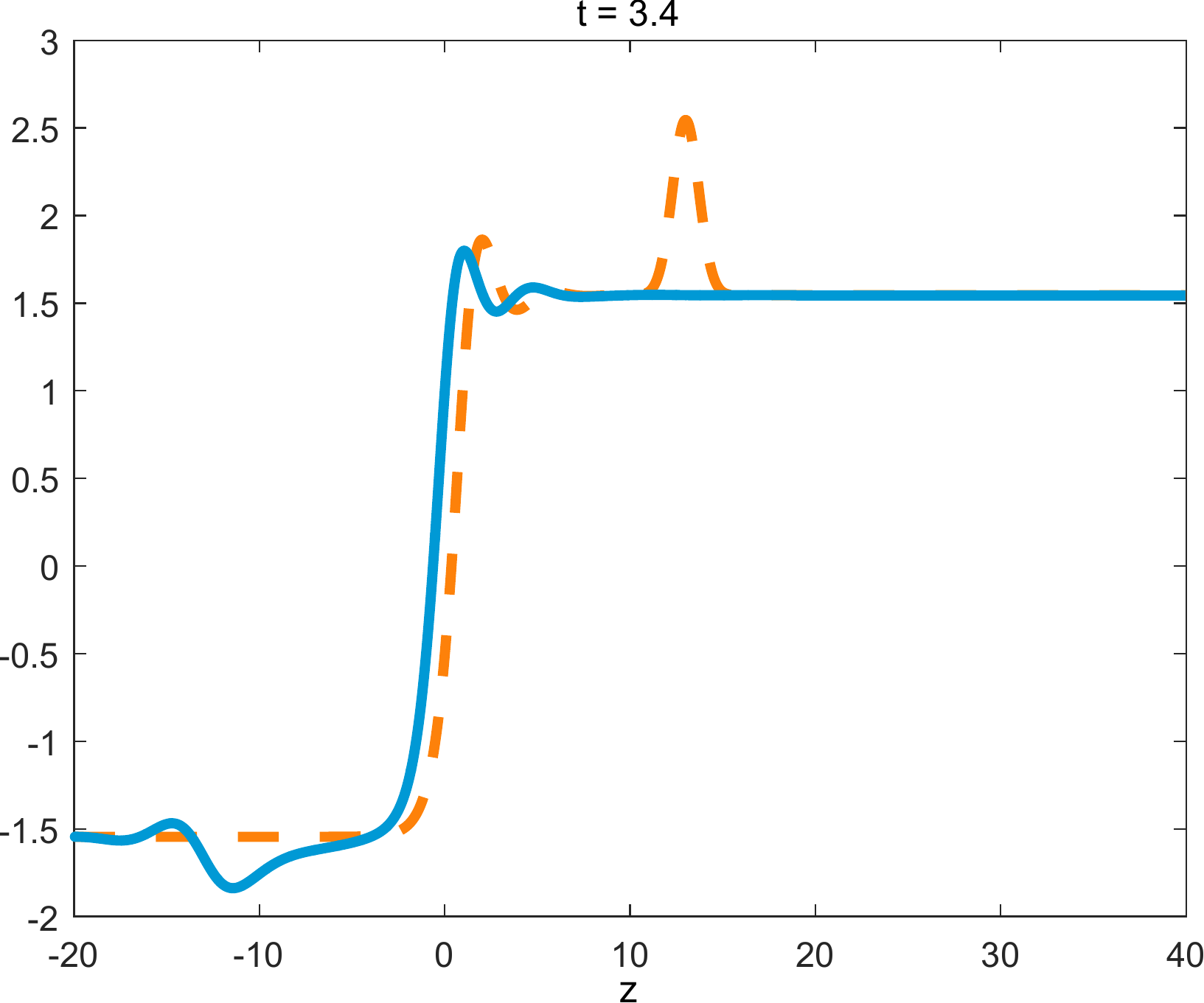}
~~(c)\hspace{-.3em}\includegraphics[scale=0.29]{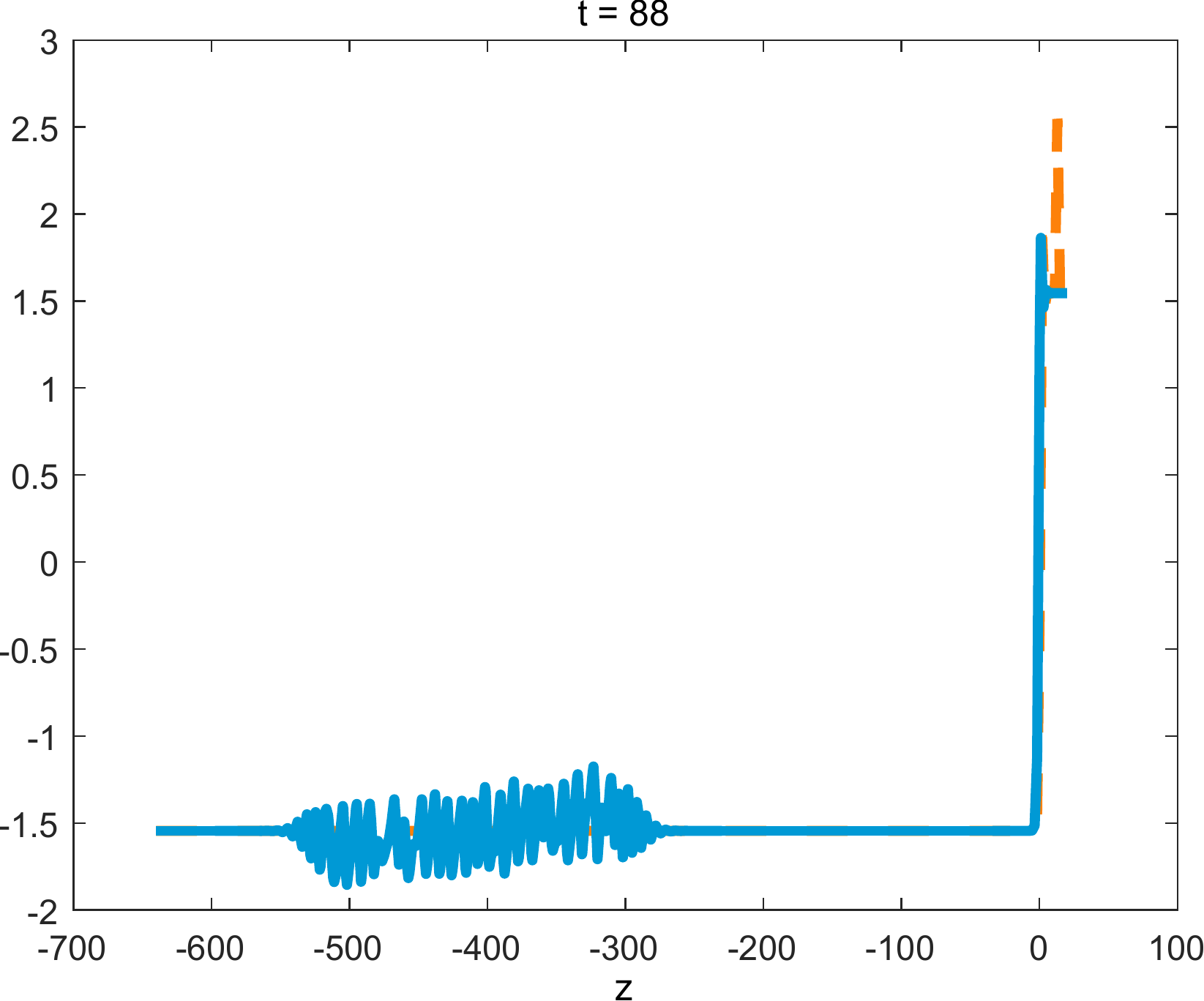}
\caption{Numerical time evolution of the $\mu=0$ front, depicted in Figure \ref{fig:Shooting-Method-Front}, when subject to a localized
perturbation.  In (a), we show the initial condition, consisting of the $\mu=0$ front and a small Gaussian perturbation to the right of the transition.
(b) The perturbed solution near the front after the perturbation has moved through the transition (now being just to the left of the transition layer).  
The perturbation has been (seemingly exponentially) damped and has slightly shifted the location of the transition, 
indicating this part of the solution may exhibit asymptotic orbital stability.  The initial condition
from (a) is shown in orange for comparison.  In (c),
we show the result of running the time evolution for a long time on a large spatial scale.  The perturbation seems to result in a bounded, but not small,
wave packet convecting to $-\infty$ with spreading, but translating, support.
}\label{evolve_front_pert}
\end{centering}
\end{figure}

However, as the perturbation continues to move to the left, beyond the seemingly stable transition part of the wave, it is clear that a sort of instability occurs: see Figure \ref{evolve_front_pert}(c).
This instability seems to convect to the left and saturates at some $\mathcal{O}(1)$ amplitude\footnote{That is, the convecting pattern saturates at some finite, but not necessarily small, amplitude.
Through repeated numerical simulations, we found that the amplitude of the convecting pattern seemed independent of the size of the initial data ahead of the transition.  This is
in stark contrast to other works where the size of the convecting pattern could be controlled by the initial perturbation, leading to a stability result of sorts: see 
\cite{BGS09,ghazaryan_nonlinear_2009,ghazaryan_nonlinear_2007}, for example.} (in $L^\infty$).  In particular, the amplitude of the perturbation does not shrink, instead
forming an ever widening wave-packet that is convected to spatial $-\infty$.  We observed the same dynamical picture in all of our numerical experiments.

From the above discussion, it is clear that the traveling fronts $\phi$ constructed in the previous section are not stable with respect to any standard (translationally invariant) Lebesgue
or Sobolev norms.  However, if we were to measure the norm of a function $v$ through a weighted norm of the form  $\|e^{\alpha z}v(z)\|$ for some appropriate $\alpha\in\RM$, 
then one might expect the convecting pattern to be negligible for large time, possibly giving rise to a stability result.  In the next sections, we aim at making this intuitive idea mathematically
rigorous.  Indeed, we shall show that, at least at the linear level, the transition layer of the front is orbitally stable (with respect to spatial translation invariance) with asymptotic phase so
long as one works in such exponentially weighted spaces.  Discussion on progress towards understanding the nonlinear dynamics near such fronts 
is given in Section \ref{s:nonlinear}.

%

\section{Stability Analysis of Front Solutions}\label{s:Essential-Instability}

In this section, we aim to mathematically describe the local dynamics near traveling front solutions constructed in Section \ref{s:front}.  We note that the existence of such
traveling front solutions remains an open problem.  In our analysis, we \emph{assume that traveling front solutions of \eqref{profile} exists}\footnote{Observe that while the 
governing equation seems similar to that studied in \cite{BS00}, the presence of the low-frequency destabilizing term $-\partial_x^4$ in \eqref{travel_pde} seems to significantly complicate
the existence theory.} in accordance with the numerical observations presented above in Section \ref{s:num_exist}.  While it would be interesting to establish such a rigorous existence theory,
here we are content with assuming such waves exist and to study their local dynamics analytically.

To this end, given a traveling front
solution $\phi$ with asymptotic values $\phi_{\pm}$ at $z=\pm\infty$, we are interested in its stability to perturbations belonging to an appropriately regular subclass 
of $L^2(\RM)$.  
Linearizing \eqref{travel_pde}  about $\phi$ in the frame of reference moving at the speed $s$, we arrive at the linear evolutionary equation
\begin{equation*}
v_t=-\partial_{z}^{2}v-\partial_{z}^{4}v+s\partial_{z}v+3\partial_{z}\left(\phi^{2}v\right):=\mathcal{L}[\phi]v,
\end{equation*}
where here $\mathcal{L}[\phi]$ is considered as a densely defined operator on $L^2(\RM)$ whose coefficients depend on the background
traveling front $\phi$.  Taking the Laplace transform in time leads to the spectral problem
\begin{equation}\label{eq:Unweighted Linearization}
\lambda v=\mathcal{L}[\phi]v
\end{equation}
considered on $L^2(\RM)$, where here $\lambda\in\CM$ is the spectral parameter.    
The wave $\phi$ is said to be \emph{spectrally unstable} if the $L^2(\RM)$-spectrum of $\mathcal{L}[\phi]$, denoted $\sigma(\mathcal{L}[\phi])$, intersects the open right half plane
of $\CM$ and it is (spectrally) stable otherwise.  

We begin by analyzing the essential spectrum $\sigma_{\rm ess}(\mathcal{L}[\phi])$ of $\mathcal{L}[\phi]$, 
defined as all $\lambda\in\CM$ where $\mathcal{L}[\phi]$ is either not Fredholm or is Fredholm with non-zero Fredholm index.  As we will see, the essential 
spectrum of $\phi$ always intersects the open right half plane of $\CM$ non-trivially, so that the traveling fronts $\phi$ constructed in Section \ref{s:front} are all
\emph{essentially unstable}.  As the essential spectrum of $\phi$ encodes the spectral stability of the asymptotic states of $\phi$, it follows
that the fronts $\phi$ are unstable in the ``far-field", away from the transition zone of the front.  We will further see that this essential instability
can be stabilized by considering $\mathcal{L}[\phi]$ on appropriate exponentially weighted space.
We will then use a combination of analytical and numerical techniques to study the point spectrum $\sigma_p(\mathcal{L}[\phi]):=\sigma(\mathcal{L}[\phi])\setminus\sigma_{\rm ess}(\mathcal{L}[\phi])$,
finding that the traveling fronts $\phi$ all have stable point spectrum.

\subsection{Analysis of the Essential Spectrum}\label{ss:ess_spec}\label{s:ess}

We begin our study of the essential spectrum of $\mathcal{L}[\phi]$ by defining the constant-coefficient asymptotic operators
\[
\mathcal{L}[\phi_{\pm}] v:=-\partial_{z}^{2}v-\partial_{z}^{4}v+\left(s+3\phi_{\pm}^2\right)\partial_zv,
\]
obtained by replacing the front $\phi$ by its asymptotic values $\phi_{\pm}$.  
Using the Fourier transform, it is readily seen that spectrum of the asymptotic operators  $\mathcal{L}[\phi_{\pm}]$ is given, respectively, by the graph of the linear dispersion relations 
\begin{equation}
p_{\pm}(ik)=-\left(ik\right)^{2}-\left(ik\right)^{4}+\left(s+3\phi_{\pm}^2\right)\left(ik\right),\label{eq:Unweighted Essential Spectrum Polynomial}
\end{equation}
over $k\in\RM$.  Clearly, the spectrum of both of the asymptotic operators $\mathcal{L}[\phi_{\pm}]$ extends into the right half (unstable) plane of $\CM$, and hence
that the asymptotic end states $\phi_{\pm}$ are always spectrally unstable: see Figure \ref{f:spec1}.  

By the Weyl essential spectum theorem, we know that the essential spectrum of the traveling front $\phi$ is determined completely by its asymptotic end states.
Specifically, the essential spectrum of $\mathcal{L}[\phi]$ corresponds to the bounded region in $\CM$ with boundaries 
determined\footnote{When $\phi_+=\phi_-$, as is the case when $\mu=0$, the set $\sigma_{\rm ess}(\mathcal{L}[\phi])$ 
consists of a countable
union of continuous curves in the complex plane.  Otherwise, $\sigma_{\rm ess}(\mathcal{L})$ is a set of non-zero two-dimensional Lebesgue measure with boundary
given by the sets $\sigma_{\rm ess}(\mathcal{L}[\phi_{\pm}])$:  see see Figure \ref{f:spec1}.} by the so-called Fredholm borders
$\sigma_{\rm ess}(\mathcal{L}[\phi_-])$ and $\sigma_{\rm ess}(\mathcal{L}[\phi_+])$, i.e.
\[
\partial\sigma_{\rm ess}\left(\mathcal{L}[\phi]\right)=\sigma_{\rm ess}\left(\mathcal{L}[\phi_-]\right)\bigcup\sigma_{\rm ess}\left(\mathcal{L}[\phi_+]\right).
\]
In particular, the essential spectrum of $\mathcal{L}[\phi]$ is contained in the stable left half plane if and only if both the sets $\sigma_{\rm ess}\left(\mathcal{L}[\phi_{\pm}]\right)$
are contained there.
From the above characterization of $\sigma_{\rm ess}\left(\mathcal{L}[\phi_{\pm}]\right)$, it follows
 that all of the traveling front solutions $\phi$ constructed in the previous section have unstable essential spectrum on $L^2(\RM)$, and are hence spectrally
unstable.  Note that this instability is consistent with the numerical time evolution studies discussed in Section \ref{s:num_dynamics_study} above, where the instability
was observed once the perturbation moved into the asymptotically constant part of the front.  
The numerics there suggest
that if we were to measure the size of perturbations in a weighted $L^2$-norm with a weight that decays sufficiently fast at $-\infty$, then the perturbation in the weighted norm
might remain small, giving us some notion of stability.

\begin{figure}[t]
 \begin{center}
(a)\includegraphics[scale=0.4]{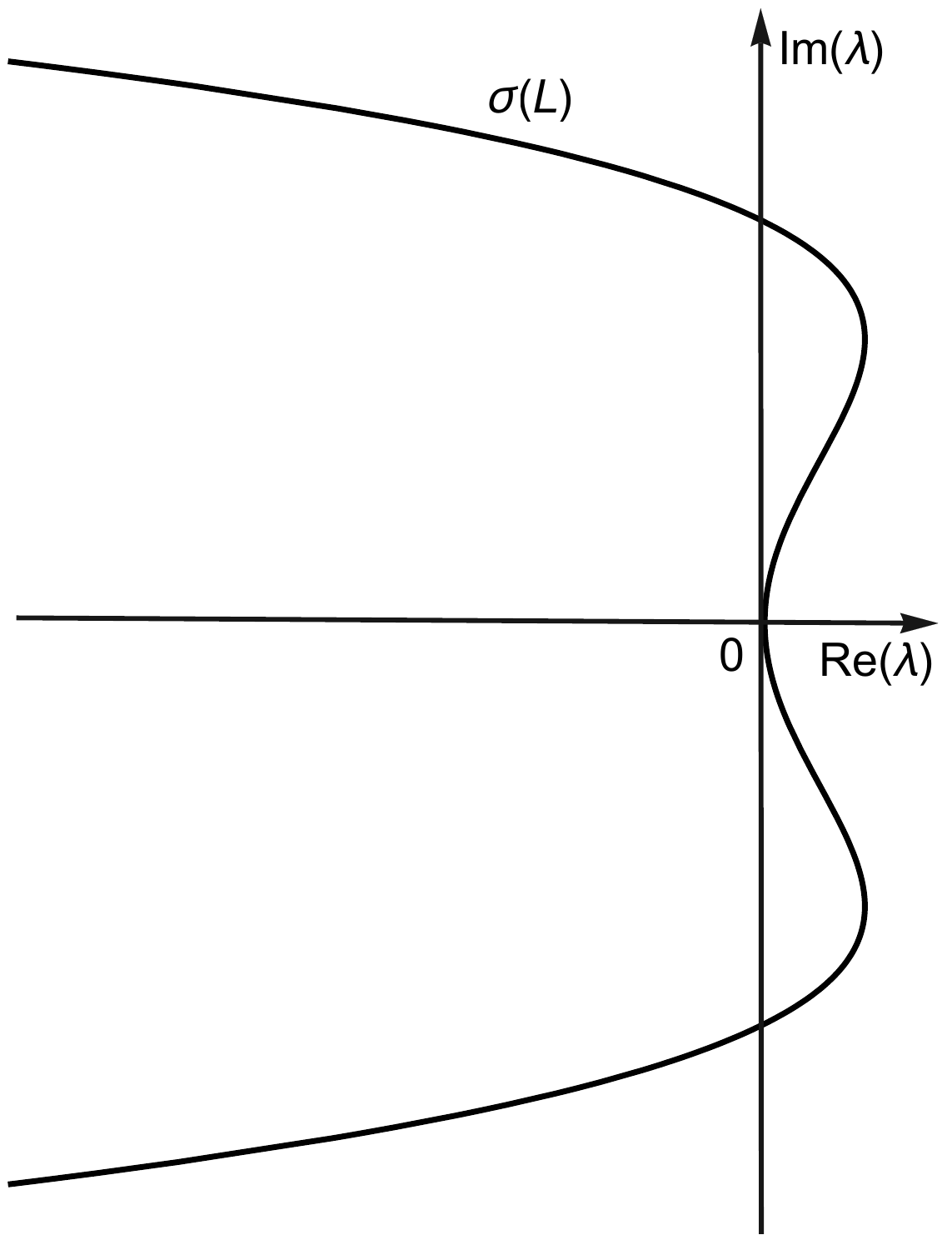}\quad\quad\quad\qquad(b)\includegraphics[scale=0.4]{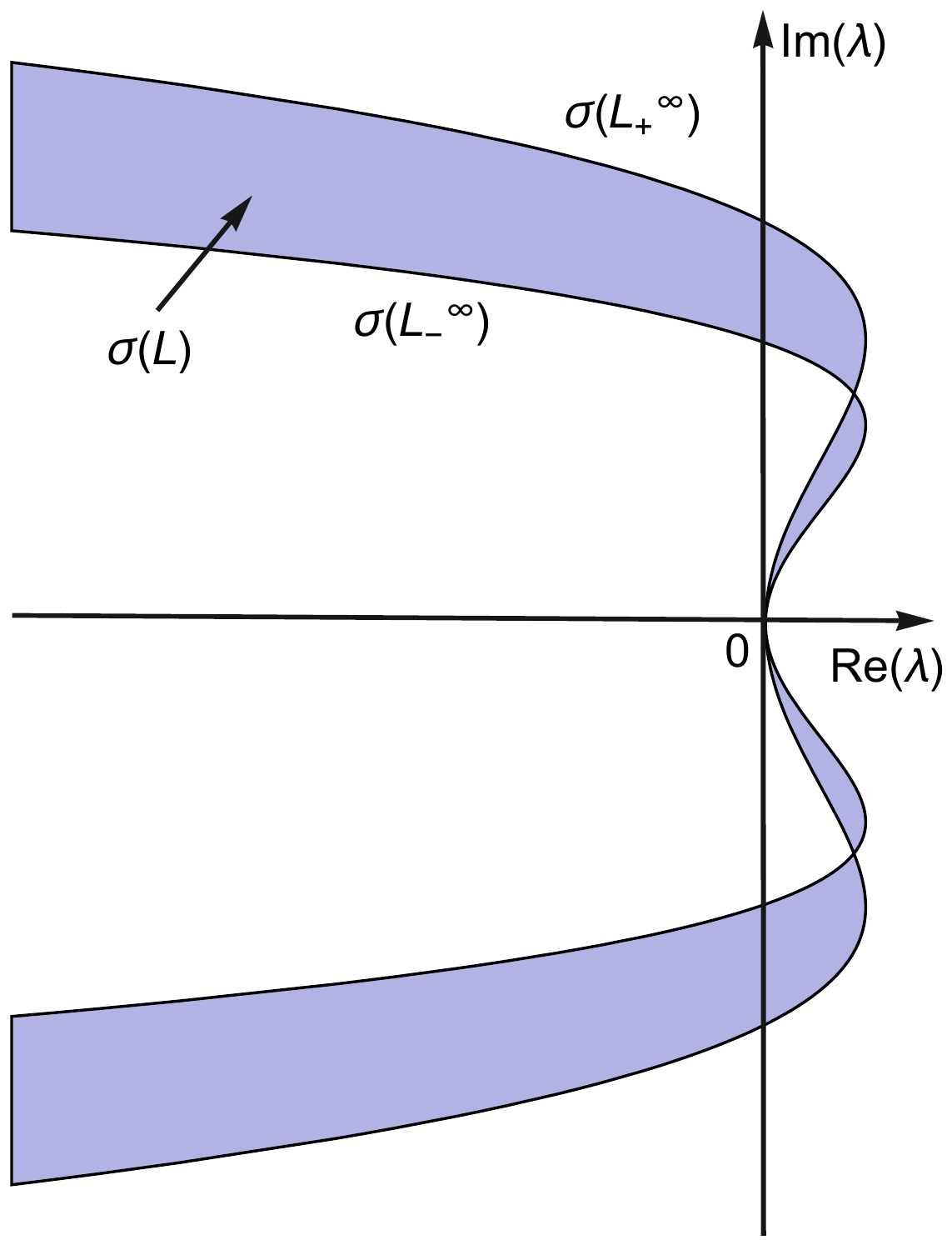}
\caption{A schematic drawing of the essential spectrum of $\mathcal{L}[\phi]$ for the traveling front $\phi$ when (a) $\phi_-=\phi_+$ and when (b) $\phi_-\neq\phi_+$.  Note in (b)
that $\sigma_{\rm ess}(\mathcal{L}[\phi])$ has non-zero two-dimensional Lebesgue measure.}  
\label{f:spec1}
\end{center}
\end{figure}
%

%

With the above observation in mind, we restrict our attention in \eqref{eq:Unweighted Linearization} to perturbations 
$v\in L^2(\RM)$ such that $e^{az}v\in L^2(\RM)$ for some $a>0$.  Setting $w(z)=e^{az}v(z)$ it follows from \eqref{eq:Unweighted Linearization} that $w$ should satisfy
\[
\lambda w=e^{az}\mathcal{L}[\phi]e^{-az}w=:\mathcal{L}_a[\phi]w,
\]
where here $\mathcal{L}_a[\phi]$ is considered as a densely defined operator on $L^2(\RM)$.  Note $\mathcal{L}_a[\phi]$ is obtained directly from $\mathcal{L}[\phi]$ via the substitution
$\partial_z\mapsto\partial_z-a$.   As above, the essential spectrum of the operator $\mathcal{L}_a[\phi]$ is bounded by the essential spectrum for the associated
asymptotic weighted operators $\mathcal{L}_a[\phi_{\pm}]$ which, in turn, is determined by the linear dispersion relation
\[
p_{\pm}(ik-a)=-(ik-a)^2-(ik-a)^4+(s+3\phi_{\pm}^2)(ik-a).
\]
Noting that the real part of the above polynomials are given by
\[
\Re\left(p\left(ik-a\right)\right)=-a^{4}+a^{2}\left(6k^{2}-1\right)-a(s+3\phi_{\pm}^2)-k^{4}+k^{2},
\]
the following result follows by elementary calculus.

\begin{figure}[t]
\begin{center}
(a) \includegraphics[scale=0.5]{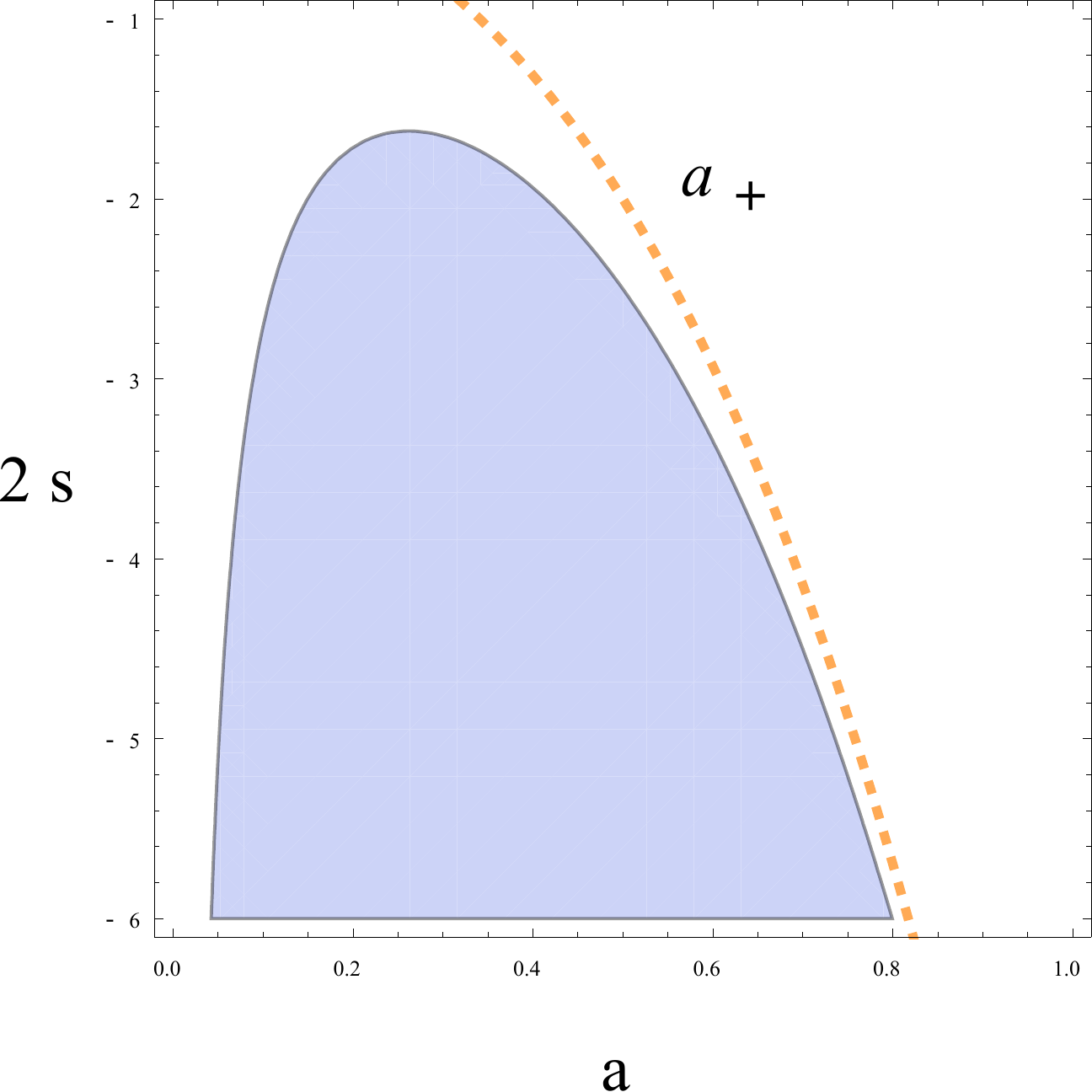}\qquad(b)\includegraphics[scale=0.6]{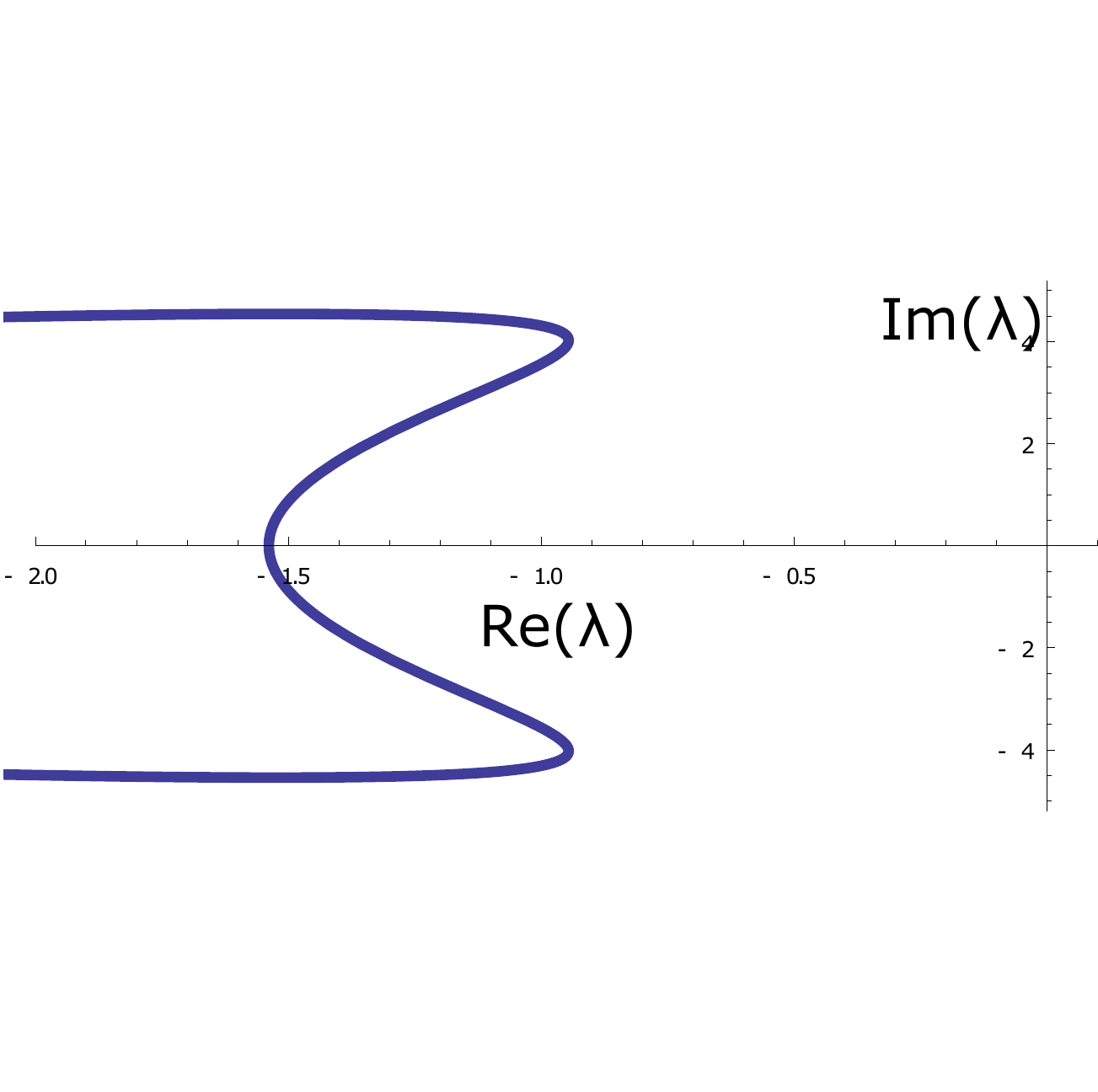}
\caption{(a) The shaded region shows the relationship between the weight $a>0$ and the wavespeed $s<0$ corresponding to the inequality in \eqref{ess_ineq} when $\mu=0$.  The inequality
\eqref{ess_ineq} holds in the shaded region.  The dashed line shows the numerically computed decay rate $a_+$ for which $e^{az}\phi'(z)\in L^2(\RM)$ for $0\leq a<a_+$
and $e^{az}\phi'(z)\notin L^2(\RM)$ for $a>a$.  (b) A plot of the essential spectrum, $\Im(\lambda)$ vs. $\Re(\lambda)$ of $\mathcal{L}_a[\phi]$ 
associated to the traveling front with $\mu=0$ and $s\approx -2.388$ and $a=0.3$.}\label{f:a_vs_s}
\end{center}
\end{figure}

\begin{lemma}\label{Lemma Essential Spectrum}
Let $\phi$ be a traveling front solution of \eqref{mks} with wave speed $s$ and asymptotic values $\lim_{z\to\pm\infty}\phi(z)=\phi_{\pm}$, and fix $a>0$.  
The essential spectrum of the weighted operator $\mathcal{L}_a[\phi]$ acting on $L^2(\RM)$ satisfies
\[
\sigma_{\rm ess}(\mathcal{L}_a[\phi])\subset\left\{\lambda\in\CM:\Re(\lambda)<0\right\}
\]
if and only if 
\begin{equation}\label{ess_ineq}
s+3\phi_{\pm}^2>\frac{32a^{4}+8a^{2}+1}{4a}.
\end{equation}
\end{lemma}

\begin{proof}
From above, the critical points of the mapping $k\mapsto \Re\left(p\left(ik-a\right)\right)$ occurs at 
\[
k=0\quad{\rm and}\quad k=\pm\sqrt{\frac{6a^2+1}{2}}.
\]
At $k=0$ we have 
\[
\Re\left(p\left(-a\right)\right)=-a^{4}-a^{2}-a(s+3\phi_{\pm}^2),
\]
which is negative for all $a$ if $s+3\phi_{\pm}^2>0$.  Similarly, at $k=\pm\sqrt{\frac{6a^2+1}{2}}$ we have
\[
\Re\left(p\left(\pm i\sqrt{\frac{6a^2+1}{2}}-a\right)\right)=8a^4+2a^2+\frac{1}{4}-a(s+3\phi_{\pm}^2),
\]
which is clearly negative provided \eqref{ess_ineq} holds.
\end{proof}

Note that from \eqref{ess_ineq}, it follows that there exists an $s_*<0$ such that the essential spectrum associated to any traveling front with wave speed $s<s_*$ can 
be stabiized by considering exponential weights $e^{az}$ for an appropriate range of $a>0$: see Figure \ref{f:a_vs_s}.  
Elementary calculus shows that $s_*\approx-0.83104.$   Conversely, for a traveling front with wave speed $s<s_*$ there exists $0<a_{\rm min }(s)<a_{\rm max}(s)$
such that
\[
\Re\left(\sigma_{\rm ess}\left(\mathcal{L}_a[\phi]\right)\right)<0\quad\forall a\in(a_{\rm min}(s),a_{\rm max}(s)).
\]
Note that $a_{\rm min}(s)>0$ since here the essential spectrum of $\mathcal{L}[\phi]$ crosses into the unstable right half plane.  This is in contrast to related studies
where the essential spectrum just touches the imaginary axis and is otherwise stable, in which case $a_{\rm min}(s)=0$: see, for example, 
\cite{ghazaryan_nonlinear_2009,ghazaryan_stability_2010}.
Furthermore, from Figure \ref{f:a_vs_s} we see
that $a_+(s)>a_{\rm max}(s)$.  As we will see in the next section, this numerical observation implies that $\lambda=0$ is an eigenvalue
of the weighted operator $\mathcal{L}_a[\phi]$  with eigenfunction $e^{az}\phi'$ for all $a$ satisfying \eqref{ess_ineq}.
It would be interesting to interpret these values $s_*$, $a_{\rm min}(s)$ and $a_{\rm max}(s)$ 
in light of the existence theory related to the ODE \eqref{profile}.

\begin{remark}\label{r:spec}
In the related works \cite{BGS09,ghazaryan_nonlinear_2007}, the essential spectrum likewise extends into the unstable right-half plane.  However, in those cases the maximum
positive real part of the essential spectrum can be made arbitrarily small by adjusting an appropriate bifurcation parameter.  This is not possible in the present case, a difficulty
which has important effects pertaining to the nonlinear dynamics.  For more discussion, see Section \ref{s:nonlinear} below.
\end{remark}


\subsection{Analysis of the Point Spectrum}\label{s:point}

Now that we know the essential spectrum of the linearized operator $\mathcal{L}[\phi]$ can be stabilized by considering exponentially weighted perturbations,
we now turn our attention to trying to analyze the point spectrum associated to the weighted operator $\mathcal{L}_a[\phi]$.  To this end, it is important to understand the relationship
between the eigenvalues of $\mathcal{L}[\phi]$ and $\mathcal{L}_a[\phi]$.  Specifically, noting that, formally, we have that
\[
\mathcal{L}[\phi]v=\lambda v\quad \iff\quad  \mathcal{L}_a[\phi]\left(e^{az}v\right)=\lambda e^{az}v,
\]
it follows that formal solutions of $\mathcal{L}[\phi]v=\lambda v$ that grow at $-\infty$ may possibly become eigenfunctions of $\mathcal{L}_a[\phi]$.  In particular,
eigenvalues of $\mathcal{L}_a[\phi]$ may appear in the unstable right half plane as the essential spectrum moves into the stable left-half plane.
Conversely, it is possible that eigenvalues of $\mathcal{L}[\phi]$ may be lost by moving to exponentially weighted spaces if the rate of convergence of the associated eigenfunction
at $+\infty$ is slower than that required by the weighted space.  Consequently, we will study the eigenvalues of the weighted operator $\mathcal{L}_a[\phi]$ for a range
of values $a$ such that the essential specturm is stable.  

We begin establishing the following result, which shows that any possible unstable eigenvalues of $\mathcal{L}_a[\phi]$ must lie in a compact region in the complex plane.
%

\begin{lemma} \label{Lemma Point Spectrum}
Suppose $\lambda\in\CM$ is an eigenvalue of $\mathcal{L}_a[\phi]$ acting on $L^2(\RM)$ with nonnegative real part.
Then $\lambda$ must satisfy the following estimates:  
\begin{equation}
\Re\lambda\leq\frac{1}{2}\norm{\partial_{x}\left(\phi^{2}\right)}_{L^{\infty}}+\frac{\alpha_{2}^{2}}{4}-\alpha_{0}-as\label{eq:Eigenvalue Bound Result 1}
\end{equation}
and
\begin{equation}\label{eq:Eigenvalue Bound Result 2}
\left\{\begin{aligned}
\Re\lambda+\left|\Im\lambda\right|&\leq\frac{\left\|\alpha_{1}+\phi^{2}+s\right\|_{L^\infty}^{2}}{4-6\alpha_{3}}+\norm{\partial_{x}\left(\phi^{2}\right)-\alpha_{0}-2a\phi^{2}-as}_{L^{\infty}} \\
&\quad+\frac{1-\alpha_{3}}{2}+\frac{1}{2}\norm{\partial_{x}\left(\phi^{2}\right)-as+\frac{\alpha_{2}}{2}}_{L^{\infty}},
\end{aligned}\right.
\end{equation}
where here
\[
\alpha_{0}=a^{2}+a^{4},\qquad\alpha_{1}=2a+4a^3,\qquad\alpha_{2}=6a^{2}+1,\qquad \alpha_{3}=4a.
\]
\end{lemma}

\begin{proof}
The proof follows by an energy estimate.
Suppose $\lambda\in\CM$ is an eigenvalue of $\mathcal{L}_a[\phi]$ with $\Re(\lambda)>0$, and let $w\in L^2(\RM)$ be a corresponding eigenfunction.  Then $w$ is a non-trivial solution
of the ODE
\begin{equation}
\lambda w=-\partial_{z}^{4}w+\alpha_{3}\partial_{z}^{3}w-\alpha_{2}\partial_{z}^{2}w+\left[\alpha_{1}+\phi^{2}+s\right]\partial_{z}w
	+\left[\partial_{z}\left(\phi^{2}\right)-\alpha_{0}-2a\phi^{2}-as\right]w.\label{eq:Eigenvalue Bound Eq 1}
\end{equation}
Multiplying \eqref{eq:Eigenvalue Bound Eq 1} by $\bar{w}$, integrating, and taking the real part gives
\begin{equation}
\Re(\lambda)\|w\|_{L^2(\mathbb{R})}^2
=-\norm{\partial_{z}^{2}w}_{L^{2}}^{2}+\alpha_{2}\norm{\partial_{z}w}_{L^{2}}^{2}-\left(\alpha_{0}+as\right)\norm w_{L^{2}}^{2}
+\frac{1}{2}\int_{\mathbb{R}}\partial_{z}\left(\phi^{2}\right)\left|w\right|^{2}-\int_{\mathbb{R}}2a\phi^{2}\left| w\right|^{2}.\label{eq:Eigenvalue Bound Eq 2}
\end{equation}
Using the Sobolev interpolation bound
\begin{equation}
\norm{\partial_{z}w}_{L^{2}}^{2}\leq\frac{1}{2C}\norm{\partial_{z}^{2}w}_{L^{2}}^{2}+\frac{C}{2}\norm w_{L^{2}}^{2}\label{eq:Eigenvalue Bound Sobolev Interp}
\end{equation}
with $C=\frac{\alpha_{2}}{2}$, and recalling that $a>0$, the bound \eqref{eq:Eigenvalue Bound Result 1} follows.

Similarly, multiplying \eqref{eq:Eigenvalue Bound Eq 1} by $\bar{w}$, integrating, and taking the imaginary part gives
%
\[
\Im(\lambda)\|w\|_{L^2(\RM)}^2
=\Im\left[\left\langle -\alpha_{3}\partial_{z}^{2}w,\partial_{z}w\right\rangle +\left\langle \left(\alpha_{1}+\phi^{2}+s\right)\partial_{z}w,w\right\rangle 
	+\left\langle \left(\partial_{z}\left(\phi^{2}\right)-\alpha_{0}-2a\phi^{2}-as\right)w,w\right\rangle \right]
\]
and subsequently, by Cauchy-Schwartz,
\begin{align*}
\left|\Im(\lambda)\right|\|w\|_{L^2(\RM)}^2
&\leq \alpha_{3}\norm{\partial_{z}^{2}w}_{L^{2}}\norm{\partial_{z}w}_{L^{2}}+\norm{\alpha_{1}+\phi^{2}+s}_{L^{\infty}}\norm{\partial_{z}w}_{L^{2}}\norm w_{L^{2}}\\
&\quad+\norm{\partial_{z}\left(\phi^{2}\right)-\alpha_{0}-2a\phi^{2}-as}_{L^{\infty}}\norm w_{L^{2}}^{2}
\end{align*}
Using Young's inequality, we have
\[
\norm{\partial_{z}u}_{L^{2}}\norm u_{L^{2}}\leq\frac{D}{2}\norm{\partial_{z}u}_{L^{2}}^{2}+\frac{1}{2D}\norm u_{L^{2}}^{2}
\]
valid for any $D>0$.  Taking $D=\frac{-\alpha_{3}-2\alpha_{3}c+2c}{\norm{\alpha_{1}+\phi^{2}+s}_{L^{\infty}}}$ and combining 
this bound with \eqref{eq:Eigenvalue Bound Sobolev Interp} with $C=1$ gives
\[
\left|\Im(\lambda)\right|\|w\|_{L^2(\mathbb{R})}^2
\leq\frac{1}{2}\norm{\partial_{z}^{2}w}_{L^{2}}^{2}+\left[\left(\frac{\norm{\alpha_{1}+\phi^{2}+s}_{L^{\infty}}^{2}}{4-6\alpha_{3}}+\norm{\partial_{z}\left(\phi^{2}\right)-\alpha_{0}-2a\phi^{2}-as}_{L^{\infty}}\right)+\frac{1-\alpha_{3}}{2}\right]\norm w_{L^{2}}^{2}
\]
Returning to \eqref{eq:Eigenvalue Bound Eq 2} and using \eqref{eq:Eigenvalue Bound Sobolev Interp} with $C=1$ gives
\[
\Re(\lambda)\|w\|_{L^2(\RM)}^2
\leq-\frac{1}{2}\norm{\partial_{z}^{2}w}_{L^{2}}^{2}+\left[\frac{1}{2}\norm{\partial_{z}\left(\phi^{2}\right)}_{L^{\infty}}-as+\frac{\alpha_{2}}{2}\right]\norm w_{L^{2}}^{2}
\]
Adding the previous two bounds together and dividing by $\norm w_{L^{2}}$ gives \eqref{eq:Eigenvalue Bound Result 2}.
\end{proof}

The utility of Lemma \ref{Lemma Point Spectrum} is that given some
$\eta>0$ and weight $a>0$ satisfying \eqref{ess_ineq}, any possible unstable eigenvalue $\lambda$
of $\mathcal{L}_a[\phi]$ with $\Re\lambda>-\eta$ necessarily lies in a compact subset of $\CM$: see Figure \ref{f:HF_EvansOutput}(a) below.  Furthermore,
the bounds \eqref{eq:Eigenvalue Bound Result 1}-\eqref{eq:Eigenvalue Bound Result 2} are quantitative: given a a traveling front solution $\phi$ with wavespeed $s$ and an allowable weight $a>0$,
Lemma \ref{Lemma Point Spectrum} provides numerical bounds in the size of any possible unstable eigenvalues of $\mathcal{L}_a[\phi]$. 
This opens the door to the use of numerical methods to locate possible unstable eigenvalues of $\mathcal{L}_a[\phi]$.  While several numerical methods exist in the literature, here we 
use numerical Evans function calculations to rule out the existence of unstable eigenvalues of $\mathcal{L}_a[\phi]$.  

To discuss our Evans function calculations, we begin by writing the weighted eigenvalue problem \eqref{eq:Eigenvalue Bound Eq 1} as the first order system
\begin{equation}\label{first_system}
Y'=A_a(z,\lambda)Y,\quad '=\frac{d}{dz},
\end{equation}
where here $Y=(w,w',w'',w''')^t\in\CM^4$ and $A$ is the $4\times 4$ matrix
\[
A_a(z,\lambda):=\left(\begin{array}{cccc}
									0 & 1 & 0 & 0\\
									0 & 0 & 1 & 0\\
									0 & 0 & 0 & 1\\
									6\phi\phi'-a^2-a^4-3a\phi^2-as-\lambda & 2a+4a^3+3\phi^2+s & -6a^2-a & 4a.
									\end{array}\right),
\]
with coefficients depending on the front profile $\phi$ and the spectral parameter $\lambda$.  Note since $\phi\to\phi_{\pm}$ as $z\to\pm\infty$, the asymptotic
behavior of solutions of \eqref{first_system} is determined by the spectral properties of the limiting constant coefficient matrices
\[
A_a^{\pm\infty}(\lambda)=\lim_{z\to\infty}A_a(z,\lambda),
\]
which are given explicitly as
\[
A_a^{\pm\infty}=\left(\begin{array}{cccc}
									0 & 1 & 0 & 0\\
									0 & 0 & 1 & 0\\
									0 & 0 & 0 & 1\\
									-a^2-a^4-3a\phi_{\pm}^2-as-\lambda & 2a+4a^3+3\phi_{\pm}^2+s & -6a^2-a & 4a.
									\end{array}\right).
\]
where here $\phi_{\pm}=\lim_{z\to\pm\infty}\phi(z)$.  So long as $a$ satisfies the estimate \eqref{ess_ineq}, an easy calculation shows that for $\Re(\lambda)>0$
both the matrices $A_a^{\pm\infty}(\lambda)$ have two eigenvalues $\mu_{j,-}(\lambda;a)$ with negative real parts and eigenvectors $v_{j,-}(\lambda,a)$, $j=1,2$ , as well as
two eigenvalues $\mu_{j,+}(\lambda;a)$ with positive real parts and eigenvectors $v_{j,-}(\lambda,a)$, $j=1,2$.  It follows that the system \eqref{first_system} admits two linearly independent
solutions $Y_{1,-}(z,\lambda,a)$ and $Y_{2,-}(z,\lambda,a)$ that converge to zero as $z\to-\infty$, as well as two linearly independent solutions $Y_{1,+}(z,\lambda,a)$ and $Y_{2,+}(z,\lambda,a)$ 
that converge to zero as $z\to+\infty$, satisfying
\[
\lim_{z\to\pm\infty} Y_{j,\pm}(z,\lambda,a)e^{-\mu_{j,\pm}(\lambda,a)z}=v_{j,\pm}(\lambda,a).
\]
In order for $\lambda\in\CM$ to be an eigenvalue of $\mathcal{L}_a[\phi]$, it must be that the space of solutions of \eqref{first_system} that are bounded as $z\to-\infty$, which is spanned by 
$Y_{1,-}$ and $Y_{2,-}$, has an intersection with strictly positive dimension with the space of solutions of \eqref{first_system} that are bounded as $z\to+\infty$, which
is spanned by $Y_{1,+}$ and $Y_{2,+}$.  Consequently, it follows that $\lambda\in\CM$ will belong to the $L^2(\RM)$ spectrum of $\mathcal{L}_a[\phi]$ provided
that it is a root of the function
\[
D_a(\lambda)=\det\left(Y_{1,-}(z,\lambda,a),Y_{2,-}(z,\lambda,a),Y_{1,+}(z,\lambda,a),Y_{2,+}(z,\lambda,a)\right),
\]
where the right hand side, by Abel's theorem, may be evaluated any value of $z\in\RM$.  The function $D_a$ above is known as the Evans 
function \cite{AGJ90,Evans75,kapitula_spectral_2013}, and it is known to be an analytic
function to the right of the essential spectrum $\sigma_{\rm ess}(\mathcal{L}_a[\phi])$ and, further, in that region the roots of $D_a$ agree in both location and algebraic multiplicity
with the eigenvalues of $\mathcal{L}_a[\phi]$.  

Thanks to Lemma \ref{Lemma Point Spectrum}, it follows that unstable eigenvalues of $\mathcal{L}_a[\phi]$ may be detected by search for the roots of the
Evans function $D_a$ on a bounded set of the form $B(0,R)\cap\{\Re(\lambda)\geq 0\}$.  To this end, we utilize the analyticity of the Evans function on the 
spectral parameter $\lambda$ and observe that given a  closed curve $\Gamma$ in $\CM$ in the domain of analyticity of $D_a$, the number of roots of $D_a(\cdot)$
can be determined from the winding number of $D_a$ around $\Gamma$, i.e. the contour integral around $\Gamma$ of the logorithimic derivative (with respect to $\lambda$
of $D_a$).  It follows that the absence of eigenvalues in a particular region of $\CM$ can thus be determined by choosing a contour $\Gamma$ whose interior contains
this region.  Such integrations can be performed numerically through the use of STABLAB, a MATLAB-based numerical 
library for Evans function computation: see \cite{barker_stablab:_nodate} for details.   Below, we describe our numerical findings.

\begin{figure}[t]
\begin{centering}
(a)~\includegraphics[scale=0.45]{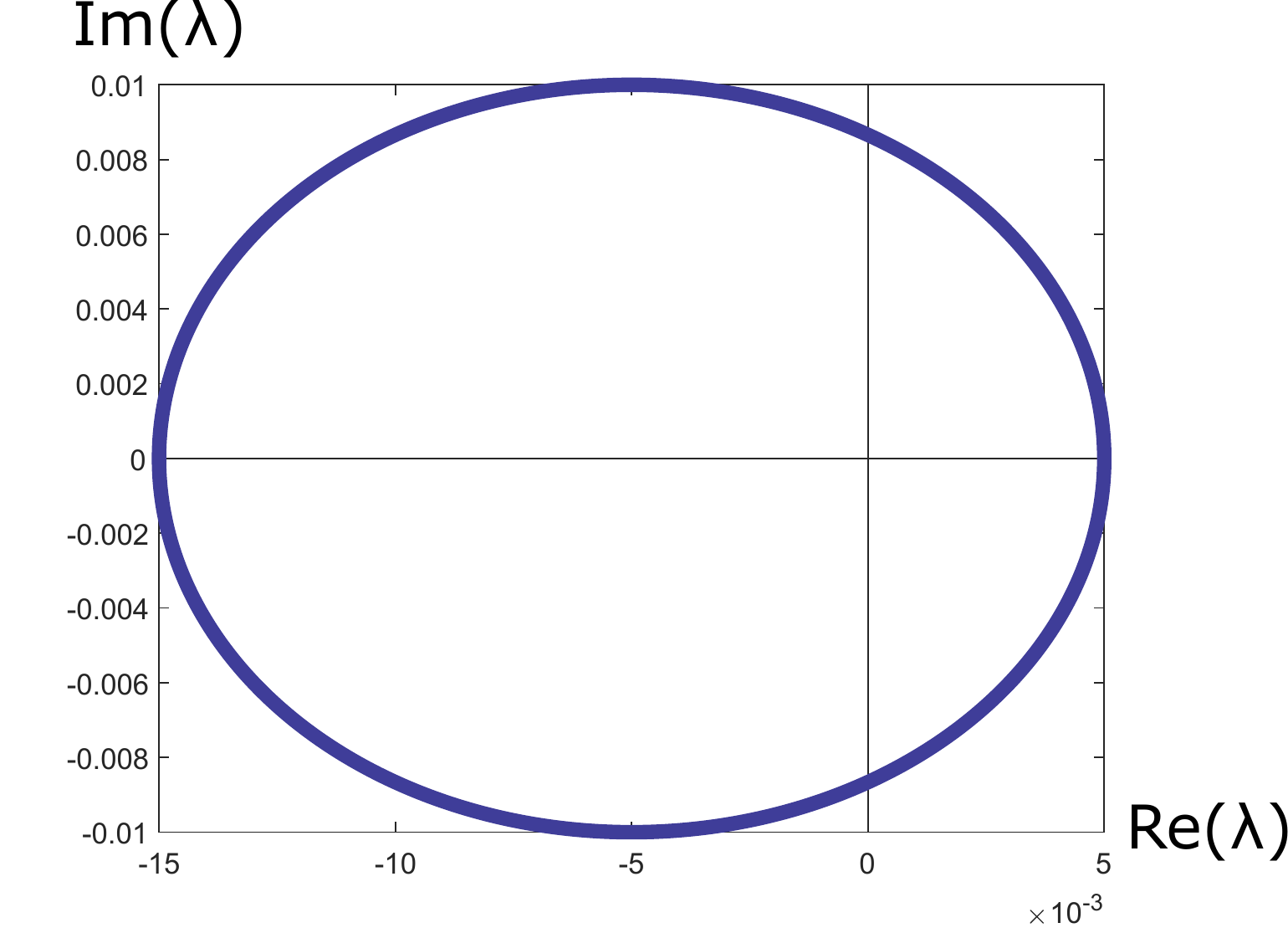}~~~(b)~\includegraphics[scale=0.45]{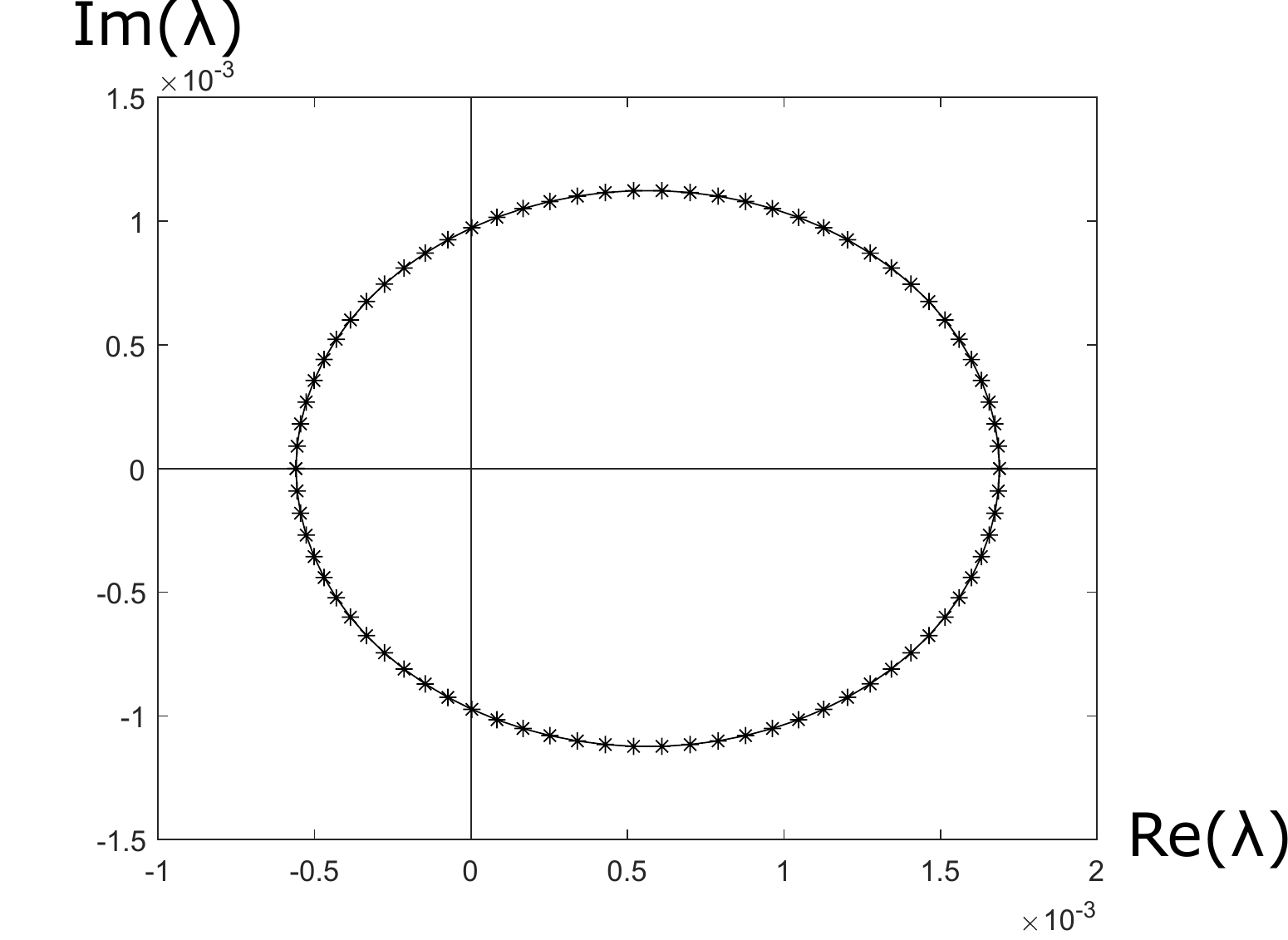}
\caption{Evans function output corresponding to the weighted linearized operator $\mathcal{L}_a[\phi]$ corresponding to $\mu=0$ and $a=0.3$.
(a) depicts the circular contour near the origin in the spectral plane, centered at $(-.005,0)$, used for the winding number calculation.  In (b) we show the output of the Evans function
when evaluated along the contour shown in (a).  Here, the winding number about the origin of the curve depicted in (b) is clearly one, implying the existence of a single eigenvalue
 (counting multiplicity) of $\mathcal{L}_a[\phi]$ being present on the interior of the contour shown in (a).}\label{f:Evans_near_origin}
\end{centering}
\end{figure}

First, we consider the case of a traveling front solution $\phi$ with $\mu=0$.    
Note that, either by direct calculation or noting that the governing evolution equation \eqref{mks} is invariant with respect to spatial translations,
$\lambda=0$ will be an (embedded) eigenvalue of $\mathcal{L}[\phi]$ with eigenfunction $\phi'$.  It follows that $\lambda=0$ will be an eigenvalue
of $\mathcal{L}_a[\phi]$ with eigenfunction  $e^{az}\phi'(z)$ provided this function still belongs to $L^2(\RM)$, i.e. provided the function $\phi'$ decays
faster at $+\infty$ than $e^{-az}$.   In particular, there exists an upper bound $a_+>0$ such that
$e^{az}\phi'\in L^2(\RM)$ for $a\in[0,a_+)$ and $e^{az}\phi'\notin L^2(\RM)$ for $a>a_+$.
Note the exponential rate of decay of $\phi'$, and hence $a_+$, can be determined directly from the existence theory, being governed by a system
of ODEs.  From Figure \ref{f:a_vs_s}(a) we find that that $e^{az}\phi'(z)\in L^2(\RM)$
for all $a$ satisfying \eqref{ess_ineq} so that, in particular, we are guaranteed $\lambda=0$ is an eigenvalue of $\mathcal{L}_a[\phi]$ for all $a$ satisfying \eqref{ess_ineq}.   
To determine the multiplicity of this eigenvalue, we numerically compute the winding number of the Evans function $D_a$ on a (sufficiently) small closed curve surrounding
the origin.  For all sampled values of $a$ satisfying \eqref{ess_ineq} this winding number was always one,  implying, in the case $\mu=0$, that $\lambda=0$ is a simple
eigenvalue of $\mathcal{L}_a[\phi]$:  see Figure \ref{f:Evans_near_origin}.  

\begin{figure}[t]
\begin{centering}
(a)\includegraphics[scale=0.35]{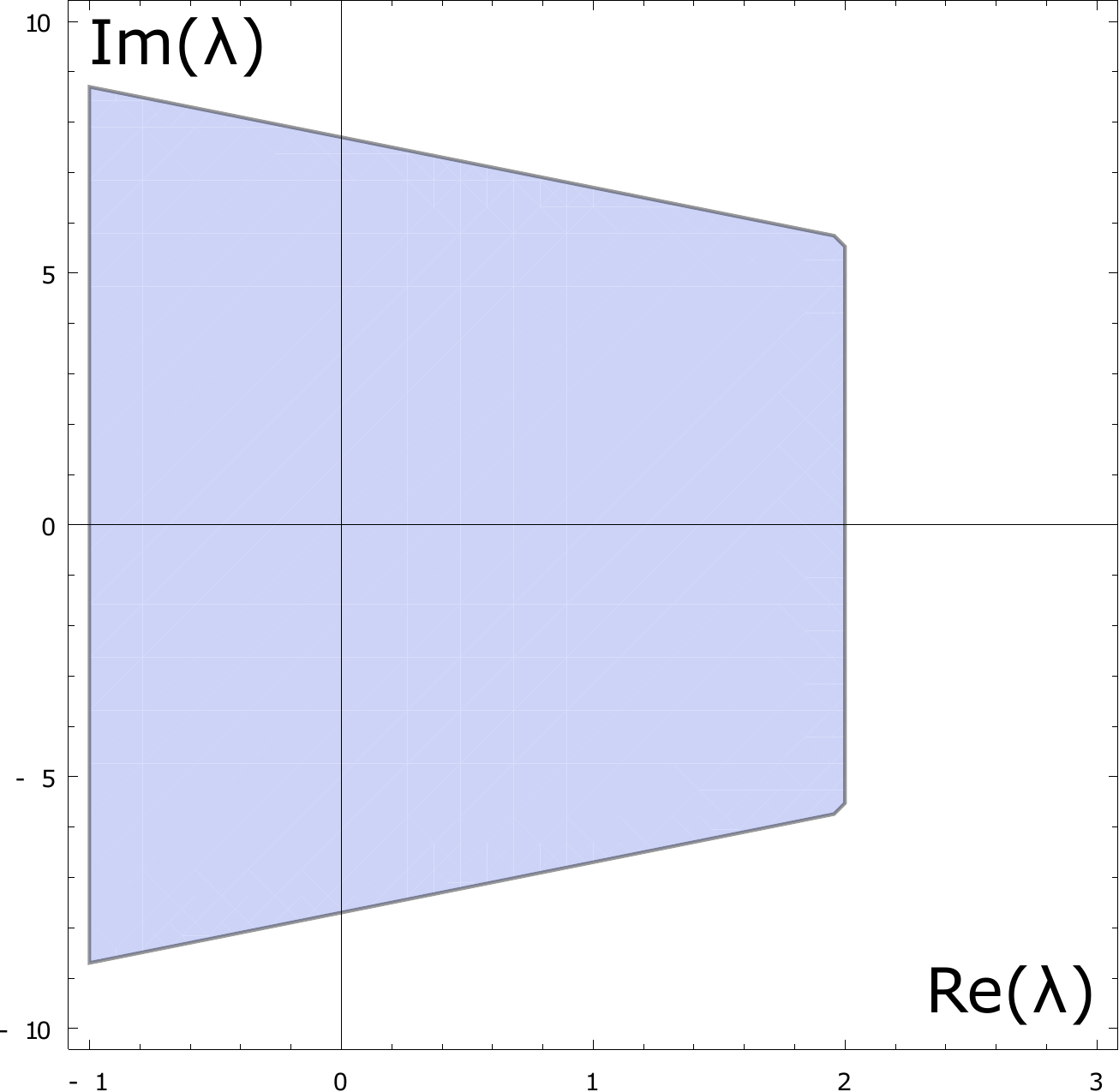}~~(b)\includegraphics[scale=0.42]{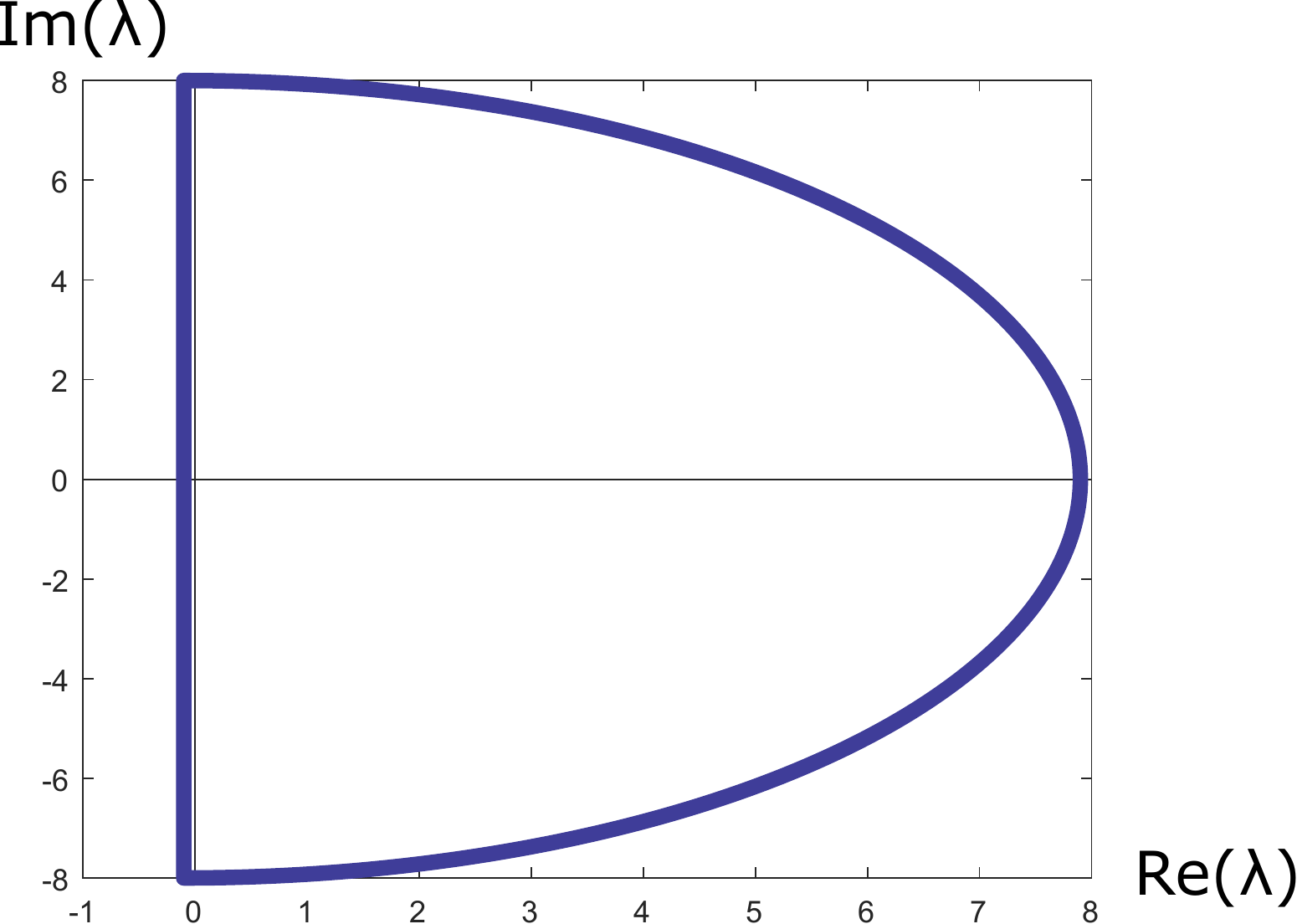}
~~(c)\hspace{-1em}\includegraphics[scale=0.42]{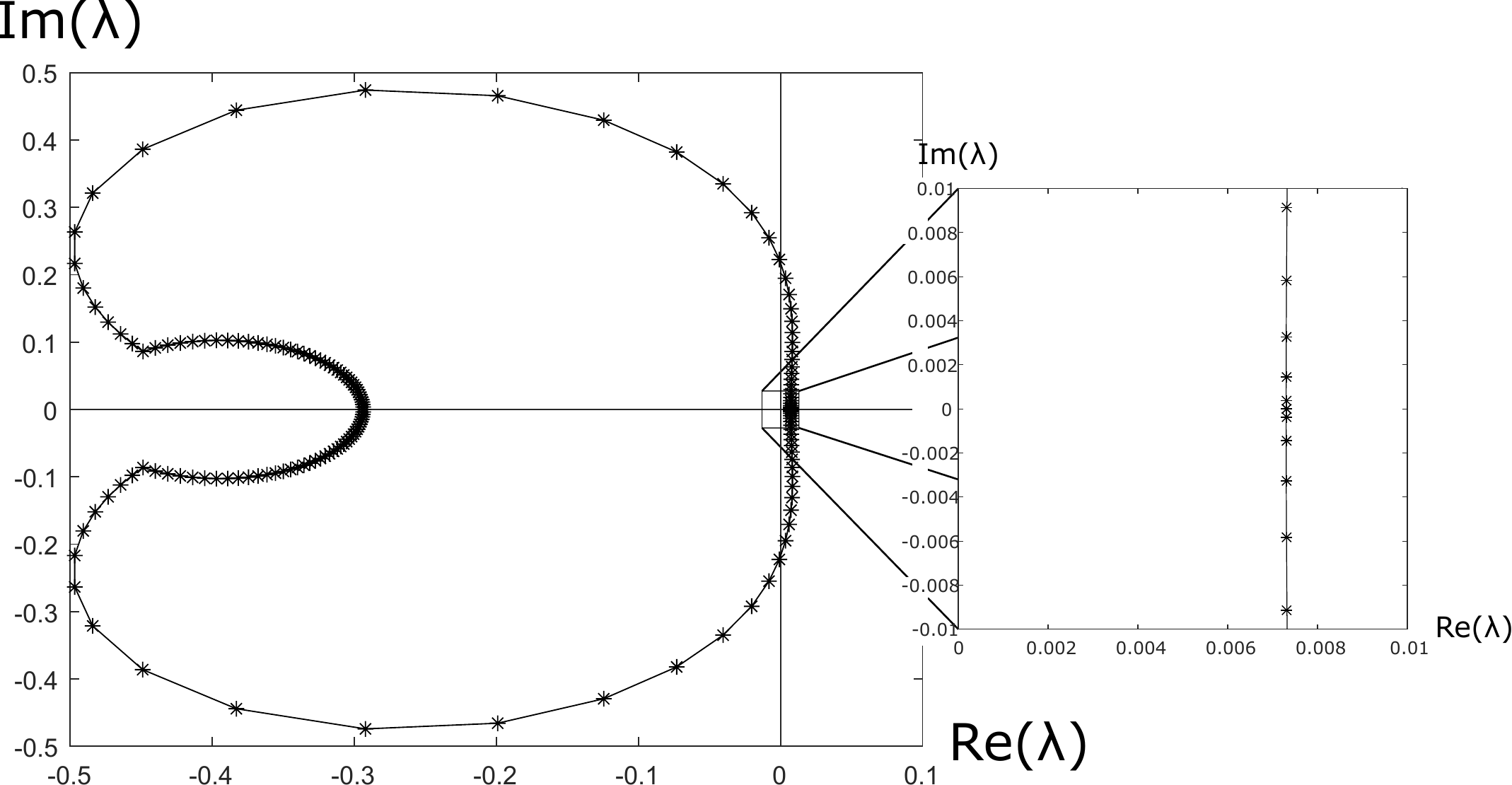}
\caption{Evans function output for the weighted linearized operator $\mathcal{L}_a[\phi]$ corresponding to $\mu=0$ and $a=0.3$.  In (a), we depict in the shaded region
the compact set in the spectral plane where unstable eigenvlaues may exist.  Figure (b) shows the contour in the spectral plane used for the winding number calculation, corresponding
to the set $\Omega$ above with $\eta=.01$ and $R=8$.  Note this contour
encloses the shaded region in depicted in (a).  Figure (c) shows the output of the Evans function $D_a$ when evaluated along the contour from (a).  Upon, inspection
the winding number of $D_a$ about the boundary of the contour in (b) is found to be one, implying the absence of unstable eigenvalues for $\mathcal{L}_a[\phi]$ in this case.}\label{f:HF_EvansOutput}
\end{centering}
\end{figure}

Continuing with studying the $\mu=0$  case, we now seek unstable eigenvalues away from the origin $\lambda=0$.  This is accomplished by 
numerically verifying that the Evans function $D_a$ has a winding number of zero on a contour of the form $\Gamma=\partial\Omega$, where here
\begin{equation}\label{omega}
\Omega=\left\{\lambda\in\CM:\Re(\lambda)\geq -\eta~~{\rm and}~~|\lambda+\eta|<R\right\}
\end{equation}
for $R$ sufficiently large and $\eta>0$ small.   In particular, we choose $R>0$ such that any possible unstable eigenavlue of $\mathcal{L}_a[\phi]$ is contained
in the interior of $\Omega$.  Using Lemma \ref{Lemma Point Spectrum}, a lower bound on the necessary size of $R$ can be determined once the traveling
wave $\phi$ is numerically computed.  Further, we choose $\eta>0$ small so that (i) the origin and (ii) any possible purely imaginary
eigenvalue of $\mathcal{L}_a[\phi]$ all contained within the interior of $\Omega$.
Our numerical Evans function calculations, achieved by numerically determining the winding number
\begin{equation}\label{winding}
\frac{1}{2\pi i}\int_{\Gamma}\frac{\partial_\lambda D_a(\lambda)}{D_a(\lambda)}~d\lambda,
\end{equation}
shows for the $\mu=0$ front and values $a$ satisfying \eqref{ess_ineq} tested that there is only one eigenvalue of $\mathcal{L}_a[\phi]$ contained in the 
interior of $\Omega$.  Since we know $\lambda=0$ is a simple eigenvalue, we conclude $\lambda=0$ is the only eigenvalue of $\mathcal{L}_a[\phi]$ satisfying
$\Re(\lambda)\geq 0$: see Figure \ref{f:HF_EvansOutput}.
It follows that the traveling front $\phi$ corresponding to $\mu=0$
has stable point spectrum, i.e.
\[
\sigma_p\left(\mathcal{L}_a[\phi]\right)\cap\left\{\Re(\lambda)\geq 0\right\}=\{0\}.
\]
Similar calculations have been performed for other values of $\mu$, likewise giving the stability of the point spectrum associated with the traveling 
front solutions $\phi$ of \eqref{mks}.  We summarize our findings as follows.

\begin{table}[t]
\begin{center}
\begin{tabular}{|c|c|c|}
\hline
$\mu$ & $s$ & $R$ \\
\hline
$-1$ & $-4.69$ & $16.39$ \\
$-.9$ & $-4.507$ & $15.91$ \\
$-.8$ & $-4.312$ & $15.4$ \\
$-.7$ & $-4.121$ & $14.87$ \\
$-.6$ & $-3.918$ & $14.29$ \\
$-.5$ & $-3.704$ & $13.67$ \\
$-.4$ & $-3.478$ & $13$ \\
$-.3$ & $-3.24$ & $12.27$ \\
$-.2$ & $-2.983$ & $11.46$ \\
$-.1$ & $-2.703$ & $10.55$ \\
$0$ & $-2.388$ & $9.39$ \\
$.1$ & $-2.016$ & $7.98$\\
$.2$ & $-1.508$ & $6.11$\\
\hline
\end{tabular}
\end{center}
\caption{A table of the values of $(\mu,s,R)$ used in the winding number calculations described in this section.  Recall the front is only expected to exist for $\mu<\mu_*\approx 0.24679$: 
see Figure \ref{fig:mu_exist}.  Here, $s$ is the wave speed associated with the travleing front $\phi=\phi(\mu)$,
and $R$ is the radius used in the definition of the contour $\Gamma=\partial\Omega$ in \eqref{omega} and \eqref{winding}.  For each of the winding number calculations,
we used $a=0.3$ and found the winding number \eqref{winding} to be equal to one,  justifying the Numerical Observation regarding the stability of the point spectrum of traveling
fronts of \eqref{mks}. }
\end{table}

\

\noindent
{\bf Numerical Observation:} \emph{Given a traveling front solution $\phi$ of \eqref{mks}, there exists a value of $a>0$ satisfying the inequality
\eqref{ess_ineq} and an $\omega=\omega(a)>0$ such that
\[
\sigma_p\left(\mathcal{L}_a[\phi]\right)\setminus\{0\}\subset\left\{\Re(\lambda)<-\omega\right\}.
\]
and $\lambda=0$ is a simple eigenvalue of $\mathcal{L}_a[\phi]$.
In particular, all such traveling wave solutions $\phi$ have unstable essential spectrum and stable point spectrum, when considered
to perturbations in (the unweighted space) $L^2(\RM)$.
}

\

Now that we have studied the spectral properties of the linearized operator $\mathcal{L}[\phi]$ in a weighted subspace of $L^2(\RM)$, we 
now aim at trying to understand the behavior of solutions of \eqref{mks} that start sufficiently close (in an appropriately weighted $L^2$-norm) to a traveling front solution.

\subsection{Linear Dynamics}\label{sec:linear}

With Lemma \ref{Lemma Essential Spectrum} and the above Numerical Observation, we now wish to upgrade the above spectral stability result, which concerns
the spectrum of the linearized operators $\mathcal{L}_a[\phi]$, to a linear stability result 
concerning the behavior of solutions of the initial value problem 
\begin{equation}\label{IVP_lin}
\left\{\begin{aligned}
&w_t=\mathcal{L}_a[\phi]w\\
&w(0)=w_0.
\end{aligned}\right.
\end{equation}
for $w_0$ belonging to some appropriately smooth subspace of $L^2(\RM)$.  Our main result in this direction is the following.

\begin{theorem}[Asymptotic Linear Orbital Stability]\label{T:main_lin}
Let $\phi$ be a traveling front solution of \eqref{mks} and suppose there exists an $a>0$ such that the following hold:
\begin{itemize}
\item[(i)] $\lambda=0$ is a simple eigenvalue of $\mathcal{L}_a[\phi]$ acting on $L^2(\RM)$ with eigenfunction $e^{az}\phi'$;
\item[(ii)] there exists a $\omega>0$ such that 
\[
\sigma(\mathcal{L}_a[\phi])\setminus\{0\}\subset\left\{\Re(\lambda)<-\omega\right\}.
\]
\end{itemize}
Then given any $w_0\in H^4(\RM)$ there exists a constant $\gamma_\infty\in\RM$ and a solution $w(z,t)$ of \eqref{IVP_lin} that is global in time and satisfies
\[
\left\|w(\cdot,t)-\gamma_\infty e^{a\cdot}\phi'\right\|_{H^2(\RM)}\lesssim e^{-\omega t}\|w_0\|_{H^2(\RM)}
\]
for all $t>0$.
\end{theorem}

Note that the assumptions (i) and (ii) above are strongly supported by our combined numerical and analytical observations from Sections \ref{s:ess} and \ref{s:point} above.
To understand the implications of Theorem \ref{T:main_lin}, we want to relate the claimed dynamics of the weighted perturbation $w(x,t)$ to the dynamics of solutions $p(z,t)$ of the original
\eqref{mks}.  To this end, suppose $v_0$ is some sufficiently small and nice function and consider the initial value problem
\begin{equation}\label{IVP_pert2}
\left\{\begin{aligned}
&p_t=sp_z-p_{zz}-p_{zzzz}+\left(p^3\right)_z\\
&p(0,z)=\phi(z)+v_0(z).
\end{aligned}\right.
\end{equation}
For so long as it exists, the solution to \eqref{IVP_pert2} can be thus decomposed as
\[
p(z,t)=\phi(z)+v(z,t)
\]
where the (unweighted) perturbation $v(z,t)$ satisfies the initial value problem
\[
v_t=\mathcal{L}[\phi]v+N(v),~~v(z,0)=v_0(z),
\]
where here $N(v)$ is at least (algebraically) quadratic in $v$ and its derivatives and satisfies $N(0)=0$.  Ignoring the nonlinear terms leads precisely to the
initial value problem \eqref{IVP_pert2}.  In terms of the unweighted perturbation $v$, Theorem \ref{T:main_lin} can be restated as follows.

\begin{corollary}\label{C:main_lin}
Under the assumptions of Theorem \ref{T:main_lin}, given any $v_0\in H^4(\RM)$ there exists a constant $\gamma_\infty\in\RM$ and a solution
$v(z,t)$ of the IVP $v_t=\mathcal{L}[\phi]v$ with $v(0)=v_0$ such that
\[
\left\|e^{a\cdot}\left(v(\cdot,t)-\gamma_\infty\phi'\right)\right\|_{H^2(\RM)}\lesssim e^{-\omega t}\|e^{a\cdot}v_0\|_{H^2(\RM)}.
\]
\end{corollary}

Consequently, saying that $e^{az}v(z,t)\approx \gamma_\infty e^{az}\phi'(z)$ for $t\gg 1$ suggests that the solution
$p(z,t)$ of \eqref{IVP_pert2} satisfies\footnote{Here, the symbol $\approx$ should be interpreted as meaning ``approximately equal
in $L^2(\RM)$".  In particular, this is NOT a pointwise approximation so that the common factor of $e^{az}$ can not simply be canceled out.  This,
of course, is evident in the numerical simulations reported in Section \ref{s:num_dynamics_study} above: see Figure \ref{evolve_front_pert}.}
\[
e^{az}p(z,t)\approx e^{az}\left(\phi(z)+\gamma_\infty\phi'(z)\right)\approx e^{az}\phi(z+\gamma_\infty)
\]
for $t\gg 1$, where the final approximation follows essentially by Taylor series.  In particular, in an exponentially weighted $L^2$ norm, we observe 
linear, asymptotic orbital stability with asymptotic phase $\gamma_\infty$.  At a linear level, this justifies the apparent asymptotic stability of the transition
component of the traveling fronts: see Figure \ref{evolve_front_pert}.

The main task in establishing Theorem \ref{T:main_lin} is to develop an exponential decay bound on the semigroup $e^{\mathcal{L}_a[\phi]t}$ acting on an appropriate subspace
of $L^2(\RM)$.  This is established by invoking the following well-known theorem: see \cite{CL03}, for instance.

\begin{theorem}[Gearhart-Pruss Theorem]
Let $B$ be the infinitesimal generator of a $\mathcal{C}_{0}$ semigroup
on the Hilbert space $Z$. Let $\omega>0$. If there exists $M>0$
so that
\[
\norm{\left(\lambda I-B\right)^{-1}}_{Z\rightarrow Z}\leq M
\]
for all $\lambda$ with $\Re\lambda>-\omega$, then
\[
\norm{e^{Bt}}_{Z\rightarrow Z}\leq e^{-\omega t}
\]
\end{theorem}

To apply the above version of the Gearhart-Pruss theorem, we must establish the uniform-boundedness of the resolvent operator $\lambda I-\mathcal{L}_a[\phi]$ acting 
on $H^4(\RM)$.  To this end, recall that in obtaining the bound \eqref{eq:Eigenvalue Bound Result 2} in Lemma \ref{Lemma Point Spectrum} above, the
penultimate step was working from the eigenvalue equation $\mathcal{L}_a[\phi]u=\lambda u$ and establishing a bound of the form
\begin{equation}
\left(\Re\lambda+\left|\Im\lambda\right|\right)\norm u_{L^{2}}^{2}\leq C\norm u_{L^{2}}^{2}\label{eq:Summary Point Spectrum Bound}
\end{equation}
for some constant $C>0$.
Repeating this procedure, fixing $\lambda\in\rho(\mathcal{L}_a[\phi])$ and working instead from the equation $\lambda u=\mathcal{L}_a[\phi]u+f$ for $f\in L^2(\RM)$, the above term becomes
\begin{equation}
\left(\Re\lambda+\left|\Im\lambda\right|\right)\norm u_{L^{2}}^{2}\leq\tilde{C}\norm u_{L^{2}}^{2}+ \left\langle f,u\right\rangle   \label{eq:Summary Point Spectrum Bound2}
\end{equation}
for some constant $\tilde{C}>0$.
Using Young's Inequality, it follows that for any $f\in L^2(\RM)$ and any solution $u\in H^4(\RM)$ of $\lambda u=L_{a}u+f$ we have
\[
\left(\Re\lambda+\left|\Im\lambda\right|\right)\norm u_{L^{2}}^{2}\leq C'\norm u_{L^{2}}^{2}+\frac{1}{2}\norm f_{L^{2}}^{2},
\]
for some constant $C'>0$.   Rearranging the above now gives the $L^2(\RM)$ bound
\[
\norm u_{L^{2}}^{2}\leq\frac{\norm f_{L^{2}}^{2}}{2\left(\Re\lambda+\left|\Im\lambda\right|-\tilde{C}\right)}
\]

The above $L^2(\RM)$ bound can be upgraded to an $H^2(\RM)$ bound by first differentiating the resolvent equation $\lambda u=L_{a}u+f$
and then repeating the above process.  For $f\in H^2(\RM)$, this leads to the $H^2(\RM)$ estimate
%
%
%
\[
\norm u_{H^{2}}^{2}\leq C_{1}\frac{\norm f_{H^{2}}^{2}}{\left(\Re\lambda+\left|\Im\lambda\right|-C_{2}\right)}
\]
for some appropriate constants $C_1,C_2>0$.  It follows for $R>0$ sufficiently large that the resolvent operator associated
to $\mathcal{L}_a[\phi]$ is uniformly bounded on the set
\[
\Omega_{\omega,R}:=\left\{\lambda\in\CM:|\lambda|>R~~{\rm and}~~\Re(\lambda)>-\omega\right\},
\]
where here $\omega>0$ is given as in the hypothesis of Theorem \ref{T:main_lin}.  Since the only element of the spectrum of $\mathcal{L}_a[\phi]$ 
within the set $\{\Re(\lambda)>-\omega\}\setminus\Omega_{\omega,R}$ is the simple eigenvalue $\lambda=0$, it follows that defining the spectral projection
\[
\Pi_a:H^2(\RM)\to{\rm ker}\left(\mathcal{L}_a[\phi]\right),~~\Pi_a g=\langle\psi,g\rangle_{L^2(\RM)}e^{a\cdot}\phi',
\]
where here $\psi$ is the unique element in the kernel of the adjoint $\mathcal{L}_a^\dag$ satisfying $\langle\psi,e^{a\cdot}\phi'\rangle=1$,
we have that the resolvent operator associated to $\Pi_a\mathcal{L}_a[\phi]$ is uniformly bounded on the set $\{\Re(\lambda)>-\omega\}$.  In particular,
it follows from the Gearhart-Pruss Theorem that
\begin{equation}\label{semigrp_expdecay}
\left\|e^{\mathcal{L}_a[\phi]t}(1-\Pi_a)v\right\|_{H^2(\RM)}\lesssim e^{-\omega t}\|v\|_{H^2(\RM)}
\end{equation}
for all $v\in H^2(\RM)$.

With the exponential decay bound \eqref{semigrp_expdecay} in hand, the proof of Theorem \ref{T:main_lin} now follows by decomposing the solution
of the IVP \eqref{IVP_lin} as
\begin{align*}
w(z,t) &= e^{\mathcal{L}_a[\phi]t}\Pi_a w_0(z) + e^{\mathcal{L}_a[\phi]t}(1-\Pi_a)w_0(z)\\
&=\langle\psi,w_0\rangle_{L^2(\RM)}e^{az}\phi'(z)+e^{\mathcal{L}_a[\phi]t}(1-\Pi_a)w_0(z)
\end{align*}
and setting $\gamma_\infty:=\langle\psi,w_0\rangle_{L^2(\RM)}$.

\subsection{Towards Nonlinear Stability}\label{s:nonlinear}

Naturally, one would like to upgrade Theorem \ref{T:main_lin} to a statement regarding the nonlinear dynamics near the traveling front solutions
of \eqref{mks}.  As we will see, nonlinear results of this type in exponentially weighted spaces do not immediately follow
from spectral, or even linear, stability due to the fact that the nonlinearity does not map exponentially weighted spaces into themselves.  
While we are unable to resolve this issue in the current context, we discuss here one approach to the problem and highlight the difficulties involved.
Establishing the nonlinear stability of traveling front solutions of \eqref{mks} in exponentially weighted spaces remains an open problem that we hope to
address in the future.

Let $\phi$ be a traveling front solution of \eqref{mks} with wave speed $s$, and consider the initial value problem \eqref{IVP_pert2}
where $v_0$ is some sufficiently small and smooth initial perturbation.  So long as it exists, Corollary \ref{C:main_lin} suggests decomposing the solution
of \eqref{IVP_pert2} as
\[
p(z,t) = \phi(z+\gamma(t))+v(z,t),
\]
where here $\gamma$ some appropriate spatial modulation to be specified\footnote{For simplicity, we assume that $\gamma(0)=0$.  This can always be taken to be true thanks to
the spatial translation invariance of \eqref{mks}.}.  Substituting this ansatz into \eqref{IVP_pert2} leads to an evolution equation of the form
\[
v_t(z,t) + \gamma'(t)\phi'(z+\gamma(t)) = \mathcal{L}[\phi]v(z,t) + \mathcal{N}\left(v(z,t),v_z(z,t),\gamma(t)\right),
\]
where here $\mathcal{N}$ is at least quadratic in its arguments.  Since the linear theory suggests decay of $v$ only in exponentially weighted norms, let $a>0$ 
satisfy \eqref{ess_ineq} and set $w(z,t)=e^{az}v(z,t)$.  Then $w$ must satisfy 
\[
w_t(z,t) + e^{az}\gamma'(t)\phi'(z+\gamma(t)) = \mathcal{L}_a[\phi]w(z,t) + e^{az}\mathcal{N}\left(e^{-az}w(z,s),e^{-az}(w_z-aw)(z,s),\gamma(s)\right).
\]
Using that $e^{az}\phi'\in{\rm Ker}\left(\mathcal{L}_a[\phi]\right)$ and that $\gamma(0)=0$, we can rewrite the above as an equivalent integral equation of the form
\begin{equation}\label{iteration}
\left\{\begin{aligned}
w(z,t) +&e^{az}\gamma'(t)\phi'(z+\gamma(t))= e^{\mathcal{L}_a[\phi]t}w(z,0)\\
&\qquad + \int_0^t e^{\mathcal{L}_a[\phi](t-s)}e^{az}\mathcal{N}\left(e^{-az}w(z,s),e^{-az}(w_z-aw)(z,s),\gamma(s)\right)ds.
\end{aligned}\right.
\end{equation}
The general program now is to write the above equation as an equivalent coupled system of equations on ${\rm Ker}\left(\mathcal{L}_a[\phi]\right)$ and 
${\rm Ker}\left(\mathcal{L}_a[\phi]\right)^\perp$, and then use the exponential decay bound \eqref{semigrp_expdecay} to solve the resulting system via nonlinear iteration:
see, for example, \cite[Section 4.3]{kapitula_spectral_2013}.

While the above procedure may seem standard, we note that a common challenge is that it is difficult to obtain bounds for $w$ in standard spaces such as $C^0$ or $H^1$
due to the exponential weight $e^{az}$.  Indeed, observe that a nonlinear term of the form $v^2v_z$ becomes
\[
e^{az}\left(e^{-az}w\right)^2\left(e^{-az}w_z-ae^{-az}w\right) = e^{-2az}w^2\left(w_z-aw\right),
\]
which may become unbounded as $z\to-\infty$.  This issue is usually handled by rewriting the above as 
\[
e^{-2az}w^2\left(w_z-aw\right)= v^2(w_z-aw)
\] and obtaining a priori estimates
in $C^0$ or $H^1$, say, to show that the unweighted perturbation $v$ remains small.  Using Sobolev embedding then, one obtains control on the
nonlinear terms, leading to nonlinear stability through an iteration argument.  This particular technique, using the interplay between spatially uniform norms and exponentially
weighted norms, was first introduced in \cite{pego_asymptotic_1994} in the context of Hamiltonian dispersive PDEs and has been used extensively in the literature
since: see, for instance, \cite{BGS09,ghazaryan_nonlinear_2009,ghazaryan_stability_2010,ghazaryan_nonlinear_2007} and references therein.  
 In each of these studies, the size of the $L^\infty$ norm
 of the unweighted perturbation $v$ may be made arbitrarily small through either the choice of initial data
for the perturbation or by choosing a bifurcation parameter sufficiently small.

In the context of \eqref{mks}, however, we are met with additional complications not present in the above references.  Namely, it seems clear from the numerical 
time evolution studies in Section \ref{s:num_dynamics_study} above that the unweighted perturbation \emph{does not remain small in in any natural Lebesgue or Sobolev norm}.  
Cases where the $H^s(\RM)$ norm does not remain small or bounded
have been handled previously in several works: see, for example, 
\cite{BGS09,ghazaryan_nonlinear_2007}.  There, the authors
use	 ``uniformly local" Sobolev spaces (introduced by Mielke and Schneider in \cite{mielke_attractors_1995}) to obtain estimates on the unweighted
perturbation.  Due to the continuous embedding of these spaces into $L^\infty$, such estimates immediately yield smallness of the unweighted perturbation in $L^\infty(\RM)$.
Further, in \cite{ghazaryan_nonlinear_2009,ghazaryan_stability_2010} direct $L^\infty$ bounds are obtained on the unweighted perturbations.
For equation \eqref{mks}, however, the numerical results in Section \ref{s:num_dynamics_study} indicate that the $L^\infty(\RM)$-norm of the unweighted perturbation
saturates at some $\mathcal{O}(1)$ level in $L^\infty$, seemingly independent of the size of the initial perturbation.  We believe this difference from the works
\cite{BGS09,ghazaryan_nonlinear_2009,ghazaryan_stability_2010,ghazaryan_nonlinear_2007} ultimately relates to the fact that there the maximum real part of the unstable essential spectrum
can be made arbitrarily small adjusting a bifurcation parameter, whereas such control is not possible in \eqref{mks}: see Remark \ref{r:spec} above.
Consequently, it is not clear to us how linear terms in \eqref{iteration} necessarily dominate the dynamics near $(v,w,\gamma)=(0,0,0)$.
We believe this is a very interesting problem that we have not seen present in the literature previously, and we consider its resolution as one of the fundamental 
open questions posed in this paper.

\section{Stabilization of Unstable Fronts}\label{sec:stabilization}

\begin{figure}
\begin{centering}
(a)\includegraphics[scale=0.4]{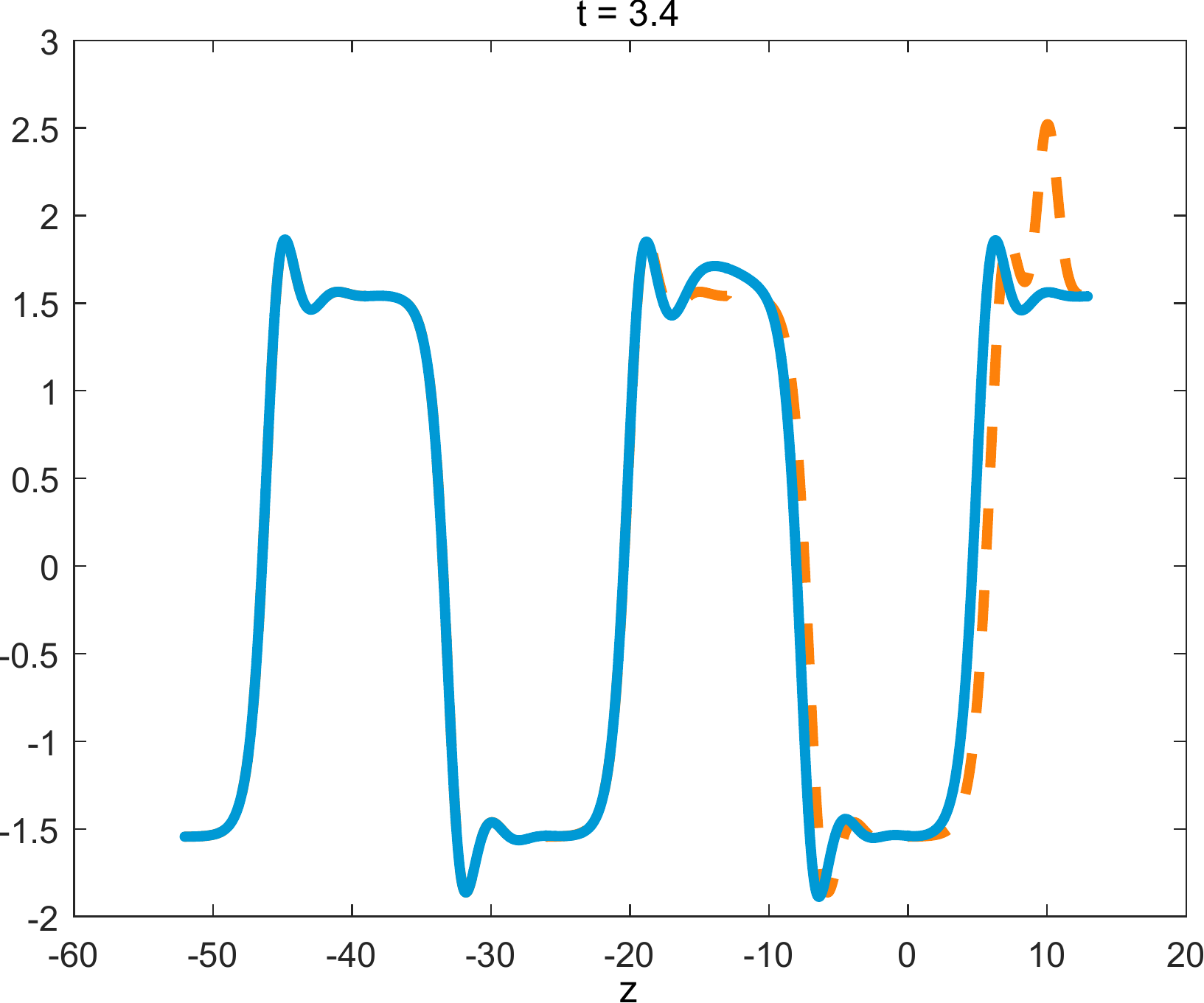}\quad
(b)\includegraphics[scale=0.4]{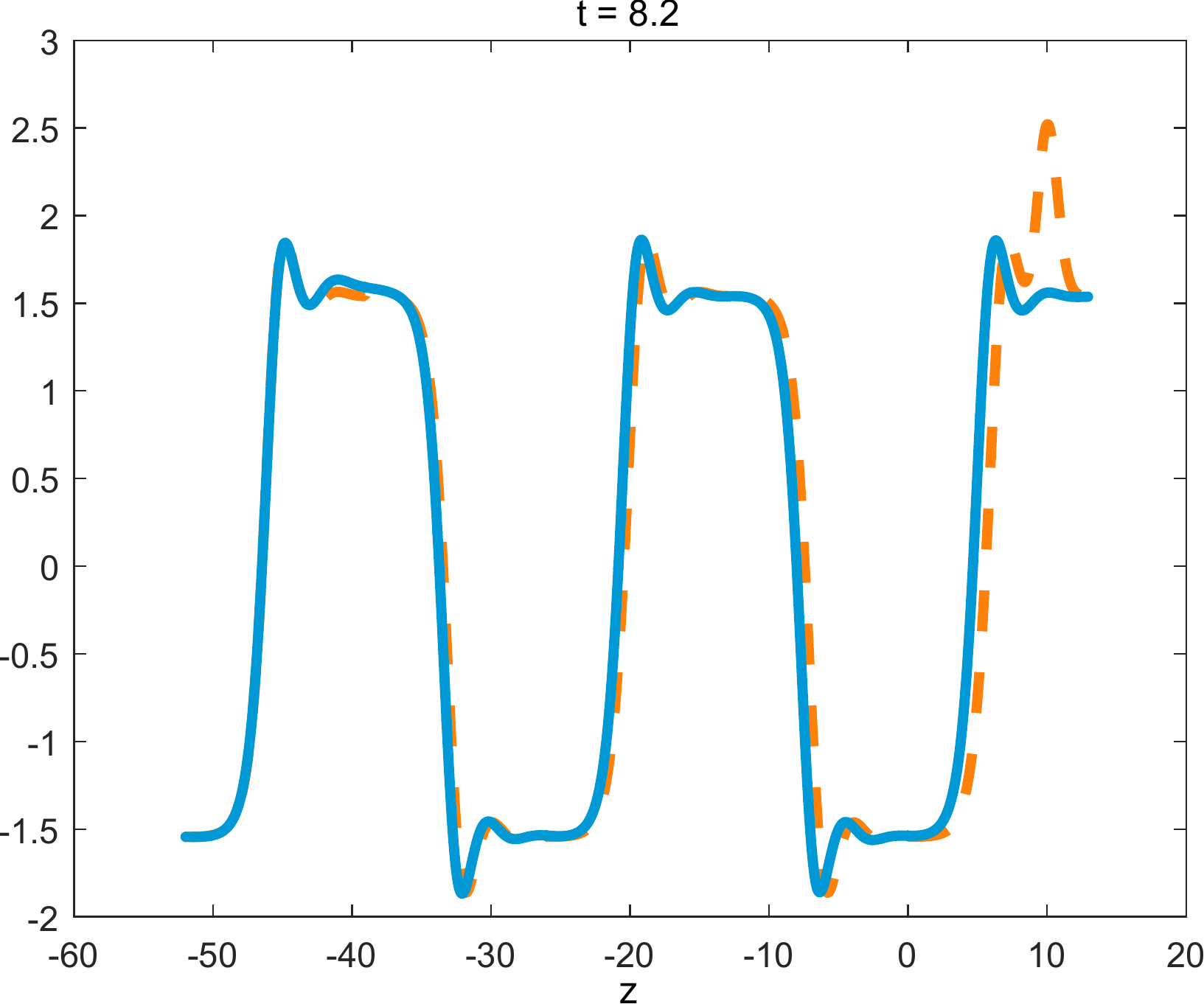}
\caption{Numerical time evolution of an ad-hoc periodic wave train formed by concatenating a $\mu=0$ front with a $\mu=0$ back solution.
The initial condition is shown in orange, and the solution at time (a) $t=3.4$ and (b) $t=8.2$ are depicted.  The periodic wave train seems 
stable, with the perturbation inducing a small phase shift on each periodic cell.  Note the size of the phase shift diminishes quickly as the perturbation
continues to convect to the left, indicating a sort of ``space-modulated" stability result in the spirit of \cite{JNRZ_Invent}.
%
}\label{f:Glued_dynamics}
\end{centering}
\end{figure}

From our previous results, especially those regarding the spectrum in Section \ref{s:Essential-Instability},
we can see that a single slope transition is not stable in an ordinary
sense. In the previous section, we described mathematically how this instability can be compensated, at least at the linear level, by introducing appropriate exponential weights.
In this section, we discuss a separate, considerably more physical stabilization mechanism.  To motivate our discussion, we note that from laboratory experiments
the surface of a nominally flat solid surface bombarded with a broad ion beam at an oblique angle of incidence generates nanoscale ripple patterns on
the decaying surface.  These ripple patterns seem comprised of a series of transitions of the surface between portions with positive and negative slopes.  
In the mathematical framework of \eqref{mks}, this corresponds to a train of front and back transitions occurring in series.  By our analysis in previous sections,
however, it is clear that each \emph{individual front and back solution is dynamically unstable}.  Consequently, a somewhat surprising phenomena arises in that
experimentally observed behavior of such nanoscale patterns is dominated by trains of front and back transitions which are themselves, in isolation, unstable.
The question now becomes to explain the observed stable behavior of trains of unstable front and back transitions.

\begin{figure}[t]
\begin{center}
(a)\includegraphics[scale=0.32]{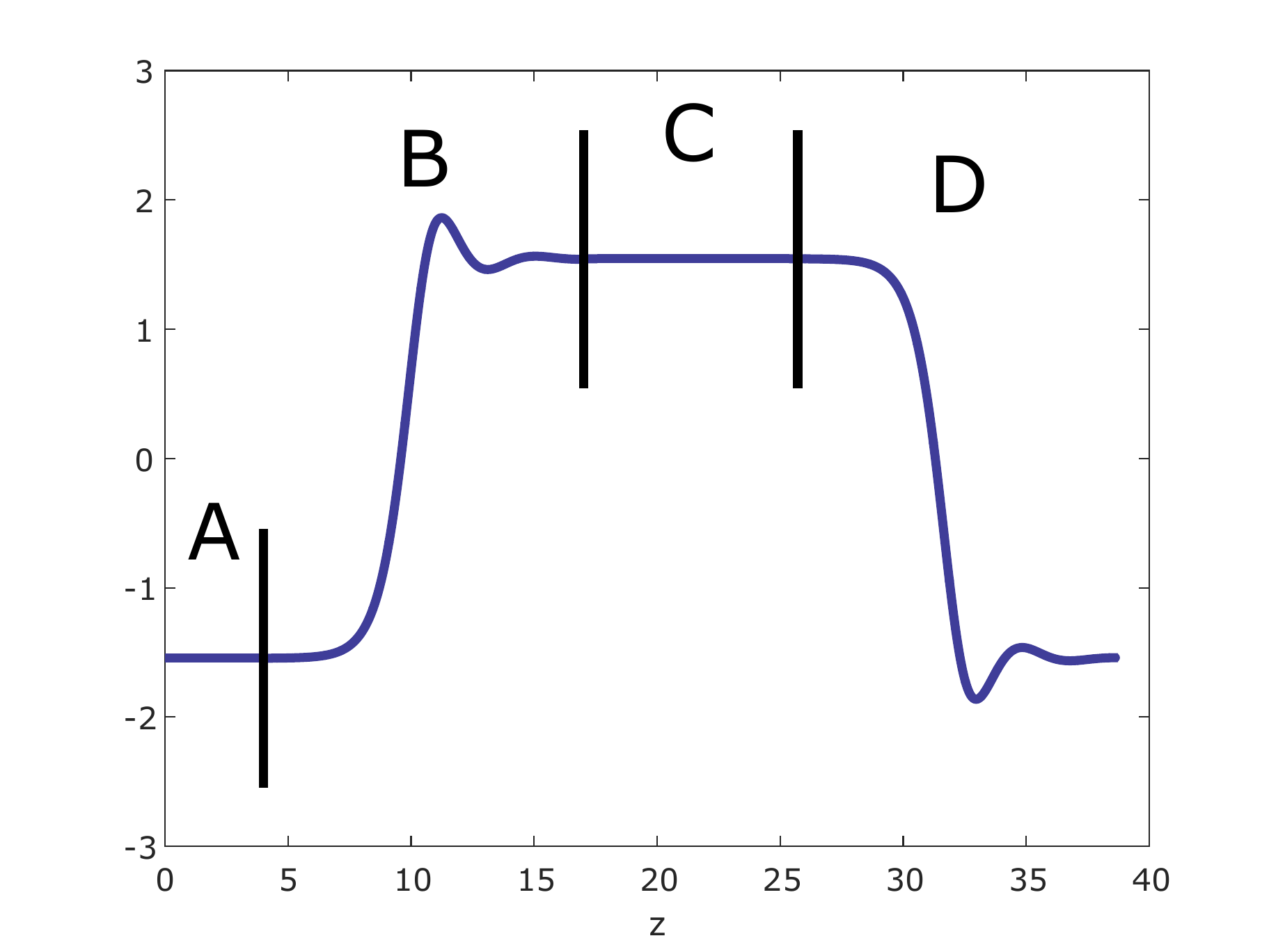}(b)\includegraphics[scale=0.4]{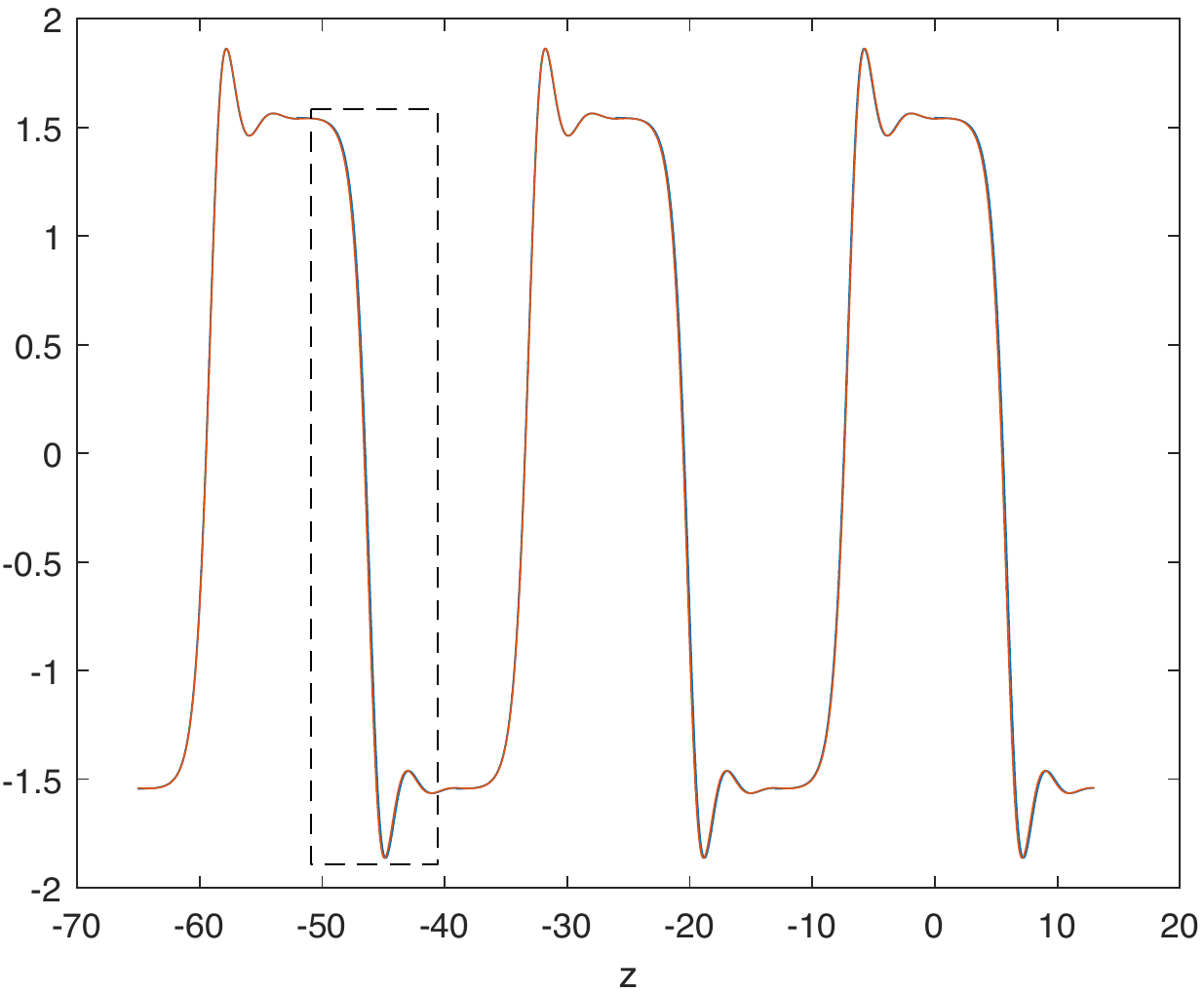}~~(c)\includegraphics[scale=0.4]{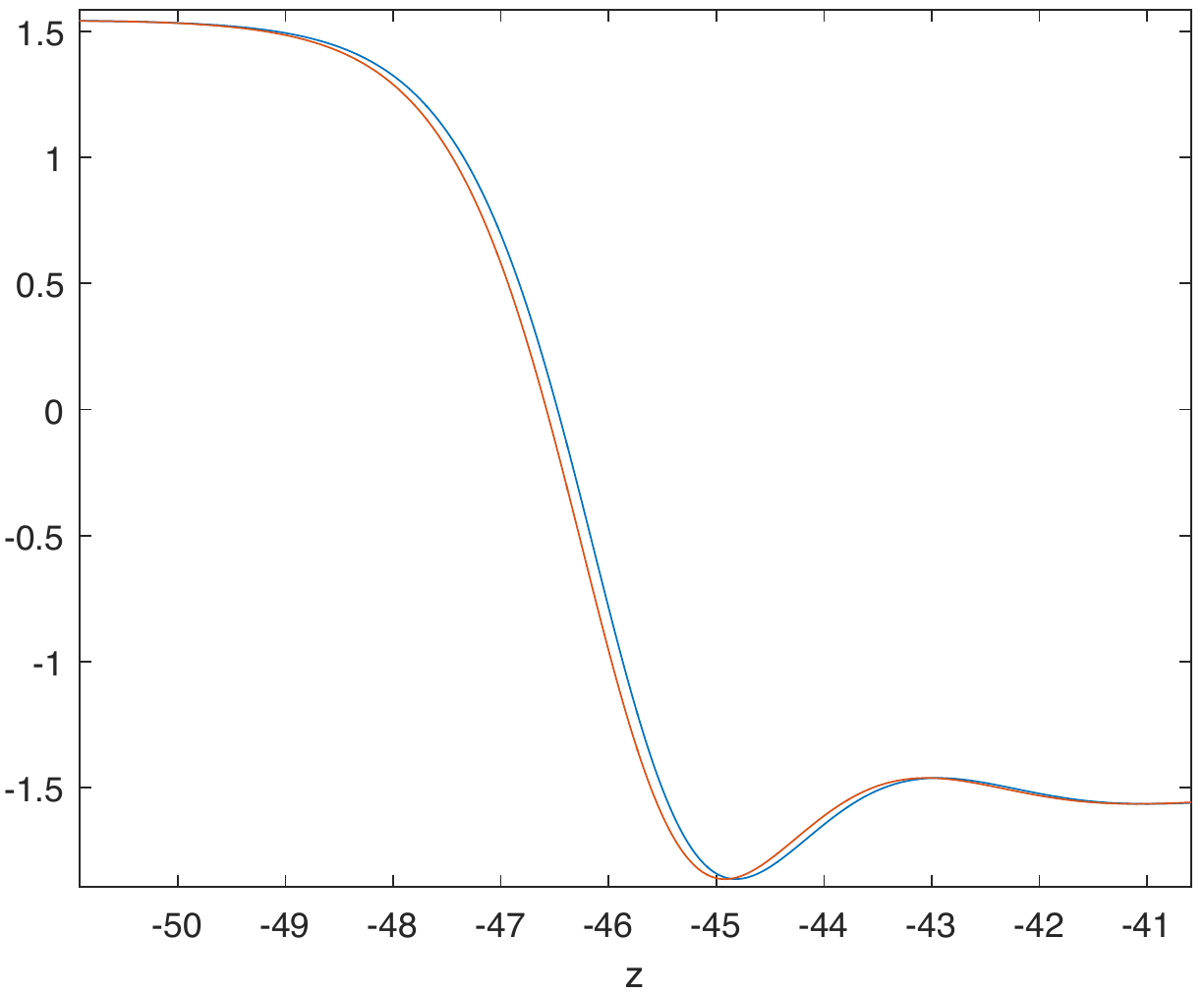}
\caption{(a) A schematic of how the ad-hoc periodic profiles are created.  Regions B and D correspond to a truncated front and back, respectively, while regions A and C correspond
to constant states with some (independently varied) spacing connecting the asymptotic end states of the fronts and backs.  (b) We superimpose the profiles of the periodically extended ad-hoc profile
considered in Figure \ref{f:Glued_dynamics}(a) with an exact periodic solution of the profile equation \eqref{profile}.  As one can see, these profiles are nearly identical.  In (c) we zoom in on the dashed
box in (b) to illustrate the close agreement of the profiles.
}\label{fig:spacing_schematic}
\end{center}
\end{figure}

By phrasing our question in the above form, the answer seems almost immediate: it must be that the local dynamics of the front and back solutions are such that
the convecting instabilities are dampened as they pass through each front or back transition, thereby counterbalancing the growth of the instability in the 
asymptotically constant portion of the wave with the stabilizing effects of the individual transitions.  We note that this stabilizing effect is readily observed in numerical
time evolution studies: see Figure \ref{f:Glued_dynamics}.

\begin{figure}[t]
\begin{centering}
\includegraphics[scale=0.5]{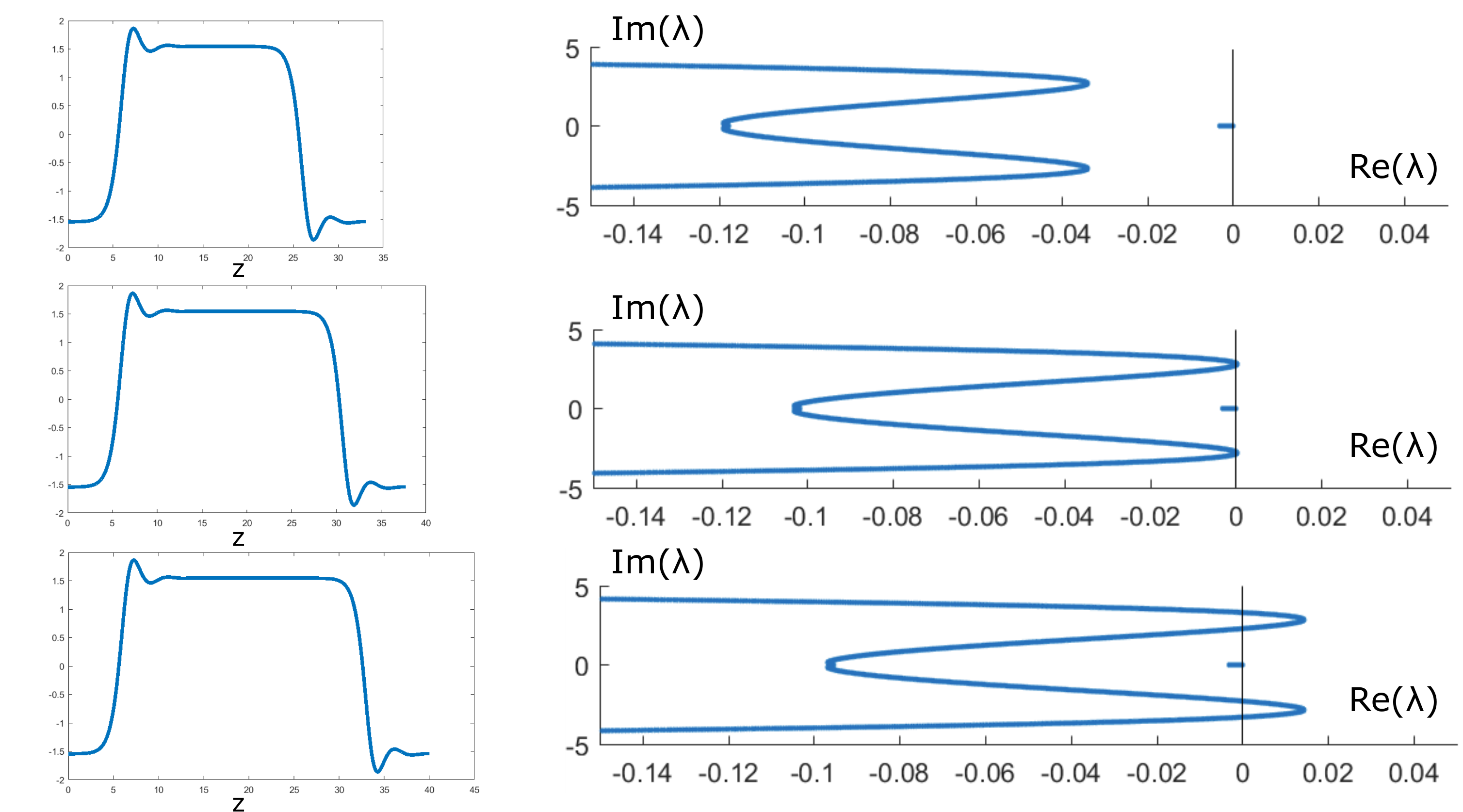}
\caption{Three different (left) front-back pairs with variable spacing and (right) the corresponding spectrum of $\mathcal{L}[\phi]$,
where $\phi$ is a periodically extended version of the pattern on the left.  In the all three profiles, the ``left spacing", corresponding to region $A$ in Figure \ref{fig:spacing_schematic}(a),
is zero.  Here, $\widetilde{X}= 7$ in the top panel, $\widetilde{X}=11.7$ in the middle panel, and $\widetilde{X}=14$ in the bottom panel.
Note it seems there is a critical amount of spacing $\widetilde{X}_{\rm crit}\approx 11.7$ between the transition layers below which the pattern is 
stable and above which it is unstable.}\label{fig:One-Spacing-Parameter}
\end{centering}
\end{figure}

In this section, we aim at studying the above phenomena by creating ad-hoc periodic patterns
of \eqref{mks} formed by concatenating a front with a back, along with some amount of space in-between them, and then periodizing this
pattern on all of $\RM$.  Specifically, we'll be constructing ``cell blocks'' by concatenating
a front, some amount of ``spacing'' which lingers on the front's
right asymptotic value, a back, more spacing which uses the back's
right asymptotic value, then possibly repeating this process with
more fronts and backs: see Figure \ref{fig:spacing_schematic}(a).  We note that the choice of truncation for each individual front and back solution is not uniquely
determined, but intuitively be made at a point where the individual front or back solution has ``almost converged" to its asymptotic 
value\footnote{Note that since $\phi_{\pm}$ correspond to hyperbolic equilibria of the profile ODE \eqref{profile}, 
all front and back profiles considered here necessarily converge at an exponential rate to their asymptotic values as $z\to\pm\infty$.}.
Naturally, each aforementioned ``spacing" can be independently varied and 
the resulting cell block can be repeated periodically to create a periodic pattern.   In Figure \ref{fig:One-Spacing-Parameter} we give
examples of the periodization of a single front-back pairs with variable spacing between them, while Figure \ref{fig:Multiple-Spacing-Parameters} shows examples
of the periodization of a cell-block formed by two front-back pairs with variable spacings\footnote{ As noted in the introduction, the construction
of such patterns is related to that constructed in \cite{SS2000}.  There, however, pulses are constructed by concatenating heteroclinic connections
between (dynamically) stable and unstable endstates, whereas here they are formed by connecting two unstable end states.}.

%
%

Once these periodic patterns are formed, we will analyze the $L^2(\RM)$-spectrum of the periodic-coefficient linear operator obtained
by substituting for $\phi$ in $\mathcal{L}[\phi]$ the above periodic pattern.  Note that while these periodic patterns are not constructed
by directly solving the profile equation \eqref{profile},
their construction and forthcoming stability properties seems to suggest they well approximate true periodic solutions.  In Figure \ref{fig:spacing_schematic}(b)-(c) we display the ad-hoc profile
considered in Figure \ref{f:Glued_dynamics} with an actual periodic solution of the profile equation \eqref{profile} 
(approximated like the heteroclinic solutions already considered using MATLAB's boundary value solver) 
with the same period and wavespeed, showing the close agreement between these
profiles.
Further, these ad-hoc solutions seem to resemble the observed phenomenon by having many alternating slope transitions. 
One major difference in that the observed physical phenomenon's slope transitions do not occur in regular intervals and hence are not periodic.
We view these ad-hoc solutions as a compromise where we can investigate irregular
spacings by changing the cell blocks, while periodically extending the cell blocks allows us to use mathematical techniques (such
as Floquet-Bloch theory) to analyze the spectrum of the resulting linear operators.
Consequently, our observations
reported below should be considered as preliminary evidence of the existence and dynamics of periodic traveling wave solutions
of the governing evolution equation \eqref{mks}.  As stated in the introduction, we consider the rigorous verification of the existence and stability of periodic
wave trains as fundamental open problems in the mathematical theory associated with the model \eqref{mks}.  Note a similar line of investigation has been
carried out in the context of inclined thin-films: see  \cite{barker_nonlinear_2013,BJRZ11,JNRZ_KS,PSU} and references therein.

One cautionary note in the construction of such ad-hoc periodic solutions is that for such a concatenation and periodization to result in a continuous pattern, the 
the asymptotic end states of the patterns being glued together must agree.  From \eqref{eq} the values of these asymptotic end states depend on the pair $(\mu,s(\mu))$, 
where $s(\mu)$ is the wave speed associated with the corresponding front/back solution of \eqref{mks}.  In order to ensure our process also generates a traveling wave,
we must also ensure that the wavespeeds of the front and back solutions being concatenated are the same.  From Figure \ref{fig:Choosing-mu}, we see that 
a front solution corresponding to a given $\mu\in\RM$ has the same speed as a back solution associated to $-\mu$.  Using \eqref{eq}, it follows that a front solution
with speed $s(\mu)$ has the same asymptotic values as a back solution with speed\footnote{Recall, as discussed in Section \ref{s:num_exist}, the symmetry 
$s_{\rm front}(\mu)=s_{\rm back}(-\mu)$.} $s(-\mu)=s(\mu)$ if and only if $\mu=0$.  Consequently,
all ad-hoc periodic patterns considered here are constructed by the concatenation and periodization of front and back solutions
corresponding to $\mu=0$.

\begin{figure}[t]
\begin{centering}
\includegraphics[scale=0.8]{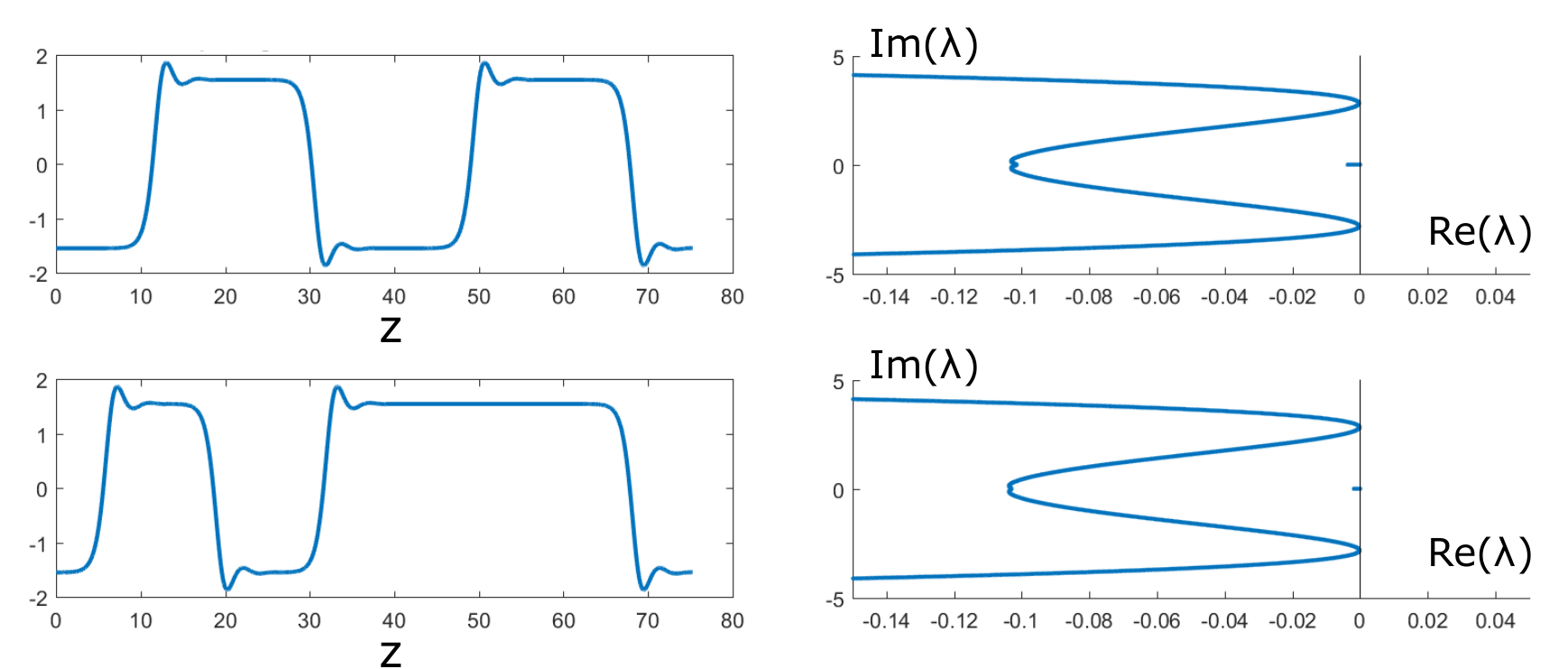}
\caption{Two different (left) cell blocks consisting of two front-back pairs with variable spacing.   In both profiles, the total spacing is $2\widetilde{X}_{\rm crit}\approx 23.2$.  In (a),
the total spacing is equally distributed into $4$ spacing regions, while in (b) the left most $3$ spacings are zero while fourth, right most spacing accounts for all of $2\widetilde{X}$.
The right of each profile shows the essential spectrum of $\mathcal{L}[\phi]$, where $\phi$ is a periodically extended version of the pattern
on the left.  Note in both cases the essential spectrum just touches the imaginary axis, corresponding to a transition between stability and instability with respect
to the total spacing $\widetilde{X}$.}\label{fig:Multiple-Spacing-Parameters}
\end{centering}
\end{figure}

In Figure \ref{fig:One-Spacing-Parameter} we depict the spectrum\footnote{The spectrum of these periodic coefficient linear operators were calculated in MATLAB using
a Fourier-Floquet-Hill method, which is essentially a Gelerkin truncation method where one expands the coefficients and unknowns into Fourier series and truncates
these series at some point, leading to a one-parameter family (parametrized by the Floquet exponent) of finite-dimensional eigenvalue problems to solve.  See
\cite{DK06,JZ12} for details.} of the periodic coefficient linearized operator $\mathcal{L}[\phi]$ about the periodization 
of several single front-back pairs, along with the periodically extended profiles used to generate the numerics.  
Note that since the linear operators considered here have periodic coefficients, standard results from Floquet theory implies that their
spectrum of $L^2(\RM)$ is entirely essential, being made up of an (at most) countable number of continuous curves in the complex plane: see  \cite{kapitula_spectral_2013} for details.
The spectrum of these operators appears to be entirely contained in the stable left half plane provided the spacing between the front and back solutions
is not too large, while it enters into the unstable right half plane once the spacing passes some critical value.  This partially justifies our intuition  that the individual
front and back transitions can mutually stabilize their instabilities when placed in an appropriately spaced wave train.

After studying these ad-hoc solutions, the general trend is that increasing
the total spacings between the consecutive transitions moves the spectrum to the right
and hence eventually makes the pattern more unstable. This can be observed in Figure \ref{fig:One-Spacing-Parameter},
where one spacing parameter for a one front-back pair is successively
increased.  Moreover, there seems to be a critical total
spacing\footnote{After first fixing a truncated front and back, corresponding to regions $B$ and $D$ in Figure \ref{fig:spacing_schematic}(a), the ``total spacing"
within a cell block corresponds to the sum of the lengths of regions $A$ and $C$ in Figure \ref{fig:spacing_schematic}(a).}
$\widetilde{X}_{\rm crit}\approx 11.7$ where the spectrum just touches the imaginary axis:
any more spacing is unstable, any less is stable. What is perhaps
more surprising is that if one builds a periodic pattern by periodically extending a cell block made out of $n$ front-back pairs,
then the critical total spacing amount differentiating stability from instability seems to be $n\widetilde{X}_{\rm crit}$.  
This is demonstrated
in Figure \ref{fig:Multiple-Spacing-Parameters}, where the spacing
of two front-back pairs is allowed to vary, but the total spacing
is kept at $2\widetilde{X}_{\rm crit}$.  In both cases depicted, the spectrum tangentially touches the imaginary
axis, signifying the transition from stability to instability of the periodic pattern.   

A very interesting consequence of the above observation is that it seems possible to construct stable periodic patterns 
in such a way where the spacing between a consecutive front and back may greatly exceed $\widetilde{X}_{\rm crit}\approx 11.7$, the critical spacing where the periodization of a 
single front back pair becomes spectrally unstable , so long as within the same
cell block is contained a sufficient number of front and back pairs with sufficiently small spacings.  This is illustrated in Figure \ref{f:bigspacing}.
It thus appears the stabilization mechanism described above is not limited to exact periodic patterns, but likely extends to
appropriately spaced arrays of front and back transitions, even if these transitions are irregularly spaced.  This observation
seems to be in line with experimental results where, as described before, the nanoscale patterns appear with irregular spacings between
slope transitions, and are hence not exactly periodic.

We end by noting the above observations strongly suggest that the modified Kuramoto-Sivashinsky equation \eqref{mks}  admits spectrally stable periodic
traveling wave solutions.  Such spectrally stable waves are known to exist in the classical Kuramoto-Sivashinsky equation (with a quadratic, rather than a cubic, nonlinearity): 
see \cite{barker_numericalproof,JNRZ_KS}.  Observe that once the existence of spectrally stable periodic traveling wave solutions of \eqref{mks} has been 
established, their nonlinear dynamics to localized and nonlocalized perturbations follows by the recent analysis in \cite{JNRZ_Invent}.  In particular, such spectrally
stable waves would necessiarly exhibit nonlinear modulational stability: see \cite{JNRZ_Invent,JNRZ_KS} for details.

\begin{figure}[t]
\begin{center}
(a)\includegraphics[scale=0.5]{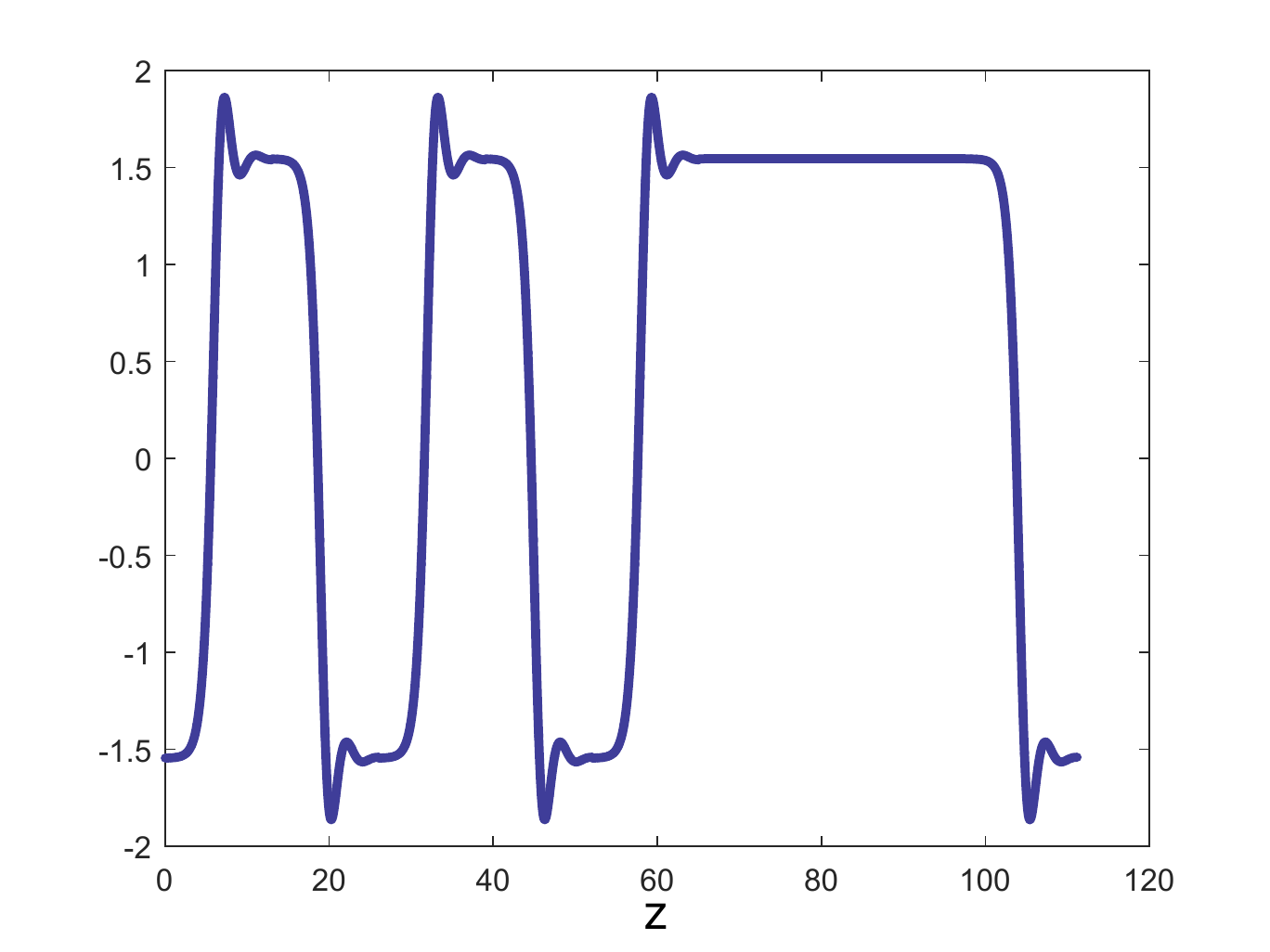}~~(b)\includegraphics[scale=0.5]{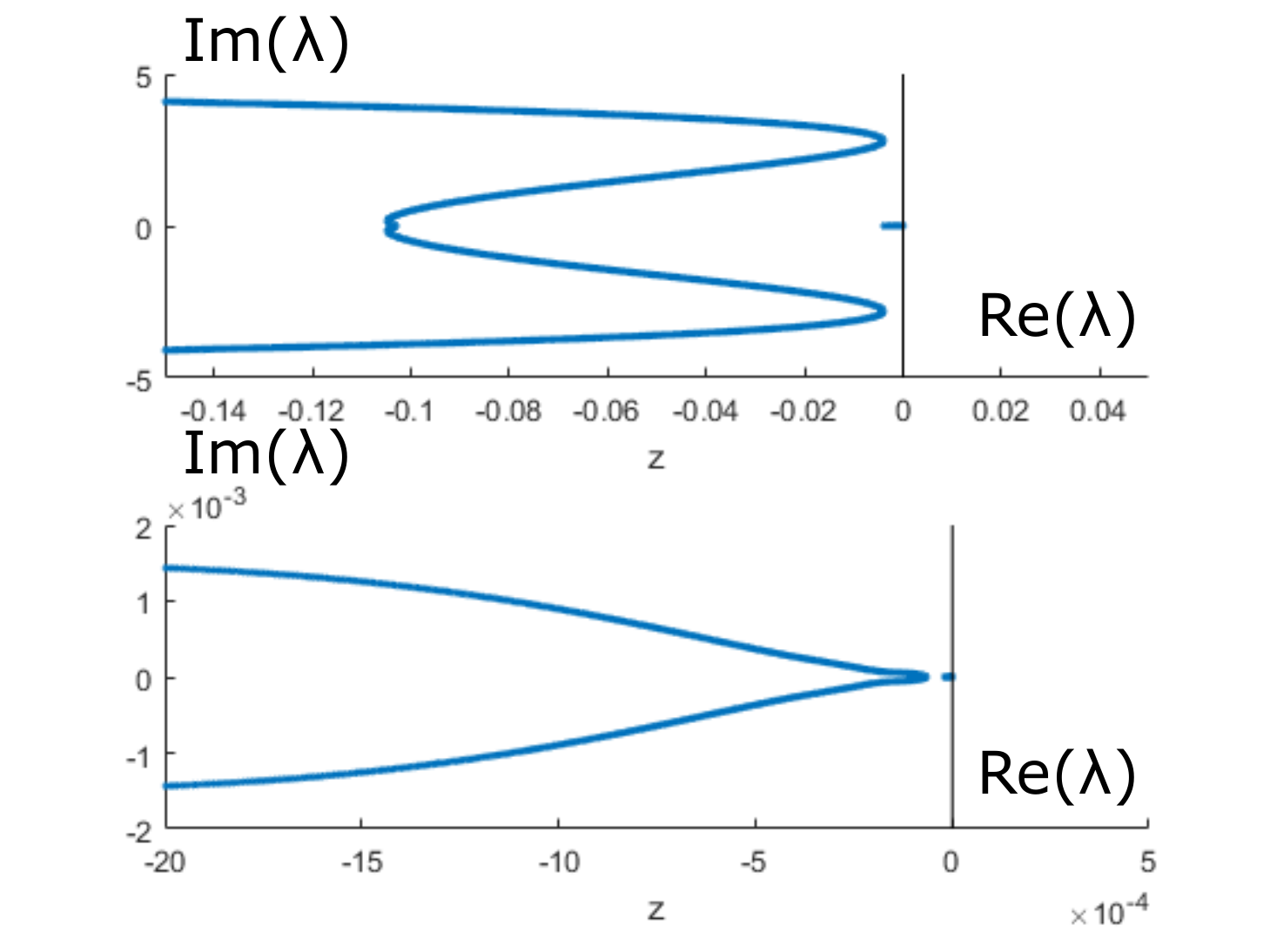}
\caption{(a) A figure of a cell block comprised of three front-back pairs with a total spacing of $\widetilde{X}=33.1$, all of which is allocated to spacing between the 
right-most front and back.  Note this spacing is well above the stability
threshold $\widetilde{X}_{\rm crit}\approx 11.7$ for a single front-back pair. (b) The spectrum of the linearized operator $\mathcal{L}[\phi]$ associated with the pattern depicted in (a): the bottom
spctral picture is a zoom on the top picture near the origin.}\label{f:bigspacing}
\end{center}
\end{figure}
%
%
%
%
%
%

\section*{Acknowledgement}
Research of MAJ was partially supported by the National Science Foundation under grant number DMS-1614785.  Research of GDL was supported in part by the National Science Foundation under grant number DMS-1413273. GDL also gratefully acknowledges illuminating conversations about this problem with R. Mark Bradley.  The authors also thank Blake Barker for helpful conversations
concerning the use of STABLAB.  Finally, the authors thank the referees for their careful reading of our manuscript, and for their many helpful suggestions and references.

%
%
\bibliographystyle{plain}
\bibliography{MetastableTravelingFronts}

\end{document}